\documentclass[11pt,reqno]{amsart}
\usepackage{amsthm, amsmath}
\usepackage{amssymb}
\usepackage{color}
\usepackage[margin=1in, footskip=0.25in]{geometry}

\usepackage[pagebackref=true]{hyperref}
\usepackage{booktabs}
\usepackage{subcaption}
\usepackage{enumerate}

\newtheorem{theorem}{Theorem}[section]
\newtheorem{lemma}[theorem]{Lemma}

\newtheorem{hypothesis}[theorem]{Hypothesis}

\newtheorem{remark}[theorem]{Remark}

\newtheorem{proposition}[theorem]{Proposition}

\newcommand{\Soc}{\mathrm{Soc}}
\newcommand{\G}{\mathrm{G}}
\newcommand{\Sy}{\mathrm{S}}
\newcommand{\A}{\mathrm{A}}
\newcommand{\C}{\mathrm{C}}

\newcommand{\Aut}{\mathrm{Aut}}
\newcommand{\Out}{\mathrm{Out}}

\newcommand{\SU}{\mathrm{SU}}
\newcommand{\GL}{\mathrm{GL}}
\newcommand{\PGL}{\mathrm{PGL}}
\newcommand{\PSp}{\mathrm{PSp}}

\newcommand{\PGU}{\mathrm{PGU}}
\newcommand{\PSU}{\mathrm{PSU}}
\newcommand{\PSL}{\mathrm{PSL}}

\newcommand{\ICF}{\mathrm{ICF}}
\newcommand{\m}{\mathrm{m}}

\newcommand{\N}{\mathrm{N}}
\newcommand{\SL}{\mathrm{SL}}
\newcommand{\AGL}{\mathrm{AGL}}
\newcommand{\ppd}{\mathrm{ppd}}

\newcommand{\POmega}{\mathrm{P\Omega}}
\newcommand{\E}{\mathrm{E}}

\newcommand{\Rad}{\mathrm{R}}
\newcommand{\Sp}{\mathrm{Sp}}

\usepackage{dynkin-diagrams}

\begin{document}
\title{Exceptional groups and the s-arc-transitivity of vertex-primitive digraphs, II}
\thanks{$^*$Corresponding author. Email address: Lei.Chen@math.uni-bielefeld.de}
\thanks{\textbf{Acknowledgements}: The authors are also grateful to Prof David Craven, Dr Mikko Korhonen, Prof Martin Liebeck and Prof Cheryl E Praeger for their precious discussions and comments. The first author is also grateful to Dr Friedrich Rober for his valuable discussions and computation assistance. The research of Lei Chen is supported by the Deutsche Forschungsgemeinschaft (DFG), German Research Foundation, Project-ID: 491392403-TRR 358. The research of Fu-Gang Yin is supported by the National Natural Science Foundation of China (12301461,12331013)
}
\author{Lei Chen\(^*\), Fu-Gang Yin}
\address{Lei Chen, Faculty of Mathematics, Bielefeld University, Bielefeld 33615, Germany.}
\address{Fu-Gang Yin,  School of Mathematics and Statistics, Beijing Key Laboratory of Biological Big Data and Topological Statistics, Beijing Jiaotong University, Beijing, P.R. China.}

\begin{abstract}
In 1989, Praeger showed that for any positive integers \(s,k\geqslant2\), there exists an infinite family of \(G\)-vertex-transitive \((G,s)\)-arc-transitive digraphs of valency \(k\). However, the existence of vertex-primitive $2$-arc-transitive digraphs remained open for nearly three decades until Giudici, Li, and Xia  constructed the first infinite family in 2017. In 2018, Giudici and Xia asked a question concerning the upper bound of $s$ for vertex-primitive $s$-arc-transitive digraphs that are not directed cycles. This question was reduced to the case where the automorphism group is almost simple, and has since been settled for several families of almost simple groups.
In this paper, we proved that, if $\Gamma$ is a \(G\)-vertex-primitive \((G,s)\)-arc-transitive digraph and $G$ is almost simple of socle \(E_7(q)\) or \(E_8(q)\), then $s\leqslant 2$.
Together with our previous work and a 2023 paper of Chen, Giudici, and Praeger, these results settle Giudici-Xia's conjecture for all exceptional groups of Lie type.

\bigskip
\noindent {\bf Key words:} $s$-arc-transitive digraph,  primitive group, almost simple group, exceptional group of Lie type, maximal subgroup.\\
{\bf 2010 Mathematics Subject Classification:}  20B25, 20D06, 05C25.
\end{abstract}
\maketitle
\tableofcontents
\section{Introduction}

A digraph $\Gamma$ is defined as a pair $(V, \to)$, where $V$ is a set of vertices and $\to$ is an irreflexive antisymmetric binary relation defined on $V$. 
For a positive integer $s$, an \emph{$s$-arc} of $\Gamma$ is a sequence of $s+1$ vertices $(v_0,v_1, \dots, v_s)$ such that $v_i\rightarrow v_{i+1}$
for each $i\in\{0,1, \dots, s-1\}$. 
For a group $G$ of automorphisms of  $\Gamma$, if $G$ acts transitively on the set of $s$-arcs, then $\Gamma$ is said to be \emph{$(G,s)$-arc-transitive}.
The terms $G$-vertex-transitive and  $G$-vertex-primitive are defined analogously, according to whether $G$ acts transitively or primitively on $V$. 

If every vertex in $\Gamma$ has both out-degree and in-degree equal to $k$, then we say that $\Gamma$ is \emph{$k$-valent}. 
Note that a $k$-valent \((s+1)\)-arc-transitive digraph is necessarily \(s\)-arc-transitive.

The study of connected vertex-primitive \(s\)-arc-transitive digraphs traces back to \cite{Praeger} in 1989, in which Praeger proved that, unlike the case for graphs (see \cite{weiss}), \(s\) can be arbitrarily large for digraphs. 
In the same paper, Praeger also noted the difficulty of constructing a vertex-primitive example for \(s\geqslant2\). 
This motivates the following two questions: the first concerns the existence of \(G\)-vertex-primitive \((G,2)\)-arc-transitive digraphs with $G$ an almost simple group (\cite[Question 5.9]{Praeger1989}), and the second asks whether \(s\) is bounded above for \(G\)-vertex-primitive \((G,s)\)-arc-transitive digraphs (\cite[Question 1.1]{quasi}).

The first question was eventually answered in \cite{example} (2017), where an infinite family of such digraphs with \(G = \PSL_3(p^2)\) and vertex-stabiliser \(G_v = \A_6\) was constructed, where \(p\equiv\pm 2\pmod{5}\) and \(p\geqslant7\).
In fact, it was proven in \cite{small} that the smallest vertex-primitive \(2\)-arc-transitive digraph also belongs to this family, having \(30,758,154,560\) vertices.

The second question has also been addressed in a series of papers. Giudici and Xia \cite{quasi} reduced the problem to the case where \(G\) is almost simple. Moreover, it was shown in \cite{CGP2023}  that \(s = 1\) when \(\Soc(G) = {}^2G_2(q)\) or \({}^2B_2(q)\), and in \cite{alter2,  CGP2024, linear,alter, ex1} that \(s \leqslant 2\) when 
\[
\Soc(G) \in \{\A_n,\PSL_n(q), \PSp_{2n}(q),G_2(2)', G_2(q), {}^3\!D_4(q),{}^2\!F_4(2)', {}^2\!F_4(q), F_4(q), E_6(q), {}^2\!E_6(q)\}.
\] 
This paper addresses the case where \(\Soc(G) = E_7(q)\) or \(E_8(q)\). Together with the aforementioned results \cite{CGP2024, ex1}, this answers the second question for all almost simple groups of exceptional type. 

The main result of this paper is the following.
 
\begin{theorem}\label{mainthm}
    Let \(G\) be an almost simple group with socle \(E_7(q)\) or \(E_8(q)\), and let \(\Gamma\) be a  \(G\)-vertex-primitive \((G,s)\)-arc-transitive digraph. Then \(s\leqslant 2\). Moreover, if \(\Gamma\) is a \(G\)-vertex-primitive \((G,2)\)-arc-transitive digraph with \(\Soc(G)\) a simple exceptional group of Lie type, then \(L_v\)
 is one of the following, as listed in Table \ref{tab:Lvcandidate} in Appendix.
 
 \end{theorem}
 
Let $G$ be an almost simple group of exceptional socle $L$.  
Despite the possibilities listed in Table \ref{tab:Lvcandidate}, it remains unknown whether such $G$-vertex-primitive $(G,2)$-arc-transitive digraphs actually exist.

In this paper, we adapt the approach of our previous work \cite{ex1}. For the parabolic case, however, we adopt a new method based on investigating factorisations of the Levi factor and analysing primitive prime divisors. The structure of maximal rank subgroups for \(E_7(q)\) and \(E_8(q)\) is also much more sophisticated than that for the smaller rank exceptional groups, involving deeper number-theoretic and representation-theoretic arguments.

\section{Preliminary}

We introduce some notation and results that will be used throughout the paper.

For a group \(G\) and a prime \(p\), we denote by  \(\mathbf{\Omega}_p(G)\), \(\mathbf{O}_p(G)\), \(\mathrm{Z}(G)\), \(\Soc(G)\), $\mathrm{F}^{*}(G)$, \(\Rad(G)\) and \(G^{(\infty)}\) the subgroup of \(G\) generated by the elements of order \(p\), the largest normal \(p\)-subgroup, the centre of $G$, the socle of $G$ (that is, the product of the minimal normal subgroups of \(G\)), the generalised Fitting subgroup of $G$, the largest soluble normal subgroup of \(G\) and the smallest normal subgroup of \(G\) such that \(G/G^{(\infty)}\) is soluble, respectively; by \(\ICF(G)\) the set of insoluble composition factors of \(G\); and by \(\m_{G}(T)\) the multiplicity of a simple group \(T\) as an insoluble composition factor of \(G\).

For two groups \(A\) and \(B\), we denote by \(A \times B\), \(A \circ B\), \(A : B\), \(A\wr B\) and \(A.B\) the direct product of \(A\) and \(B\), the central product of \(A\) and \(B\), the semidirect product of \(A\) by \(B\), the wreath product of \(A\) and \(B\), and an unspecified extension of \(A\) by \(B\), respectively.

Our notation for simple groups is standard. 
Moreover, for \(\varepsilon = \pm 1\), \(E^{\varepsilon}_6(q)\) denotes \(E_6(q)\) if \(\varepsilon = +1\) or \({}^2\!E_6(q)\) if \(\varepsilon = -1\); similarly, \(\mathrm{PSL}_n^{\varepsilon}(q)\) denotes \(\mathrm{PSL}_n(q)\) if \(\varepsilon = +1\) or \(\mathrm{PSU}_n(q)\) if \(\varepsilon = -1\). 
Lie types are  denoted by italicised letters: \(A_n, B_n, C_n, D_n, E_6, E_7, E_8, F_4, G_2\).
For a group \(X(q)\) of Lie type, we write \({X(q)}_{\mathrm{sc}}\) for its simply connected version. In particular, \({E_7(q)}_{\mathrm{sc}} = (2, q-1).E_7(q)\) is quasisimple, whereas \({E_8(q)}_{\mathrm{sc}} = E_8(q)\) is simple.

Let \(n \geqslant 2\) be a positive integer.
We use \(\mathrm{A}_n\) and \(\mathrm{S}_n\) for the alternating and symmetric groups of degree \(n\), respectively.
We write \(\mathrm{C}_n\), \(\mathrm{D}_n\), \(\mathrm{Q}_n\) for the cyclic group, the dihedral group, the generalised quaternion group of order \(n\), respectively.
We also use \(n\) as a shorthand for a cyclic group of order \(n\), and use \([n]\) for a soluble group of order \(n\) with unspecified structure.

\subsection{Number-theoretic results}
For two positive integers $m$ and $n$, their greatest common divisor is denoted by $(m,n)$.
For a finite set \(X\) and a finite integer \(n\), we denote by \(\pi(X)\) and \(\pi(n)\) the set of prime divisors of \(|X|\) and \(n\), respectively; and by \(n_{r}\) the \(r\)-part of \(n\), where \(r\) is a prime.

The next lemma is well-known, see for instance~\cite[Lemma 2.1]{linear}.  

\begin{lemma}[Legendre's formula]\label{sizeppd}
For any positive integer \(n\) and prime \(r\) we have \((n!)_{r}<r^{n/(r-1)}\).
\end{lemma}


For a prime power \(q\geqslant2\) and an integer \(n\geqslant2\), a prime number \(r\) is called a \emph{primitive prime divisor} of \(q^{n}-1\) if \(r\) divides \(q^{n}-1\) but does not divide \(q^{i}-1\) for any positive integer \(i<n\). We denote by \(\ppd(q,n)\) the set of primitive prime divisors of \(q^{n}-1\).

\begin{lemma}[Zsigmondy, \cite{Zsigmondy}]\label{existppd}
For a prime power \(q\geqslant2\) and an integer \(n\geqslant2\), the set \(\ppd(q,n)\neq \varnothing\) except when \((q,n)=(2,6)\) or \((q,n)=(2^k-1,2)\) for some integer \(k\geqslant2\).
\end{lemma}

For \(r\in\ppd(q,n)\) we derive from Fermat's Little Theorem that \(r\equiv 1\pmod{n}\); in particular, \(r>n\). For convenience, we set \(\ppd(2,6)=\{7\}\) throughout this paper.

Let \(\Phi_n(q)\) denote the \(n\)-th cyclotomic polynomial, that is, 
\(\Phi_n(q) = \prod_{\zeta} (x - \zeta)\) where $\zeta$ ranges over the primitive complex $n$-th roots of unity. 
It is well-known that \(q^n - 1 = \prod_{d\mid n} \Phi_d(q)\).
Consequently, if \(r\) belongs to \(\ppd(q,n)\), then $r \mid \Phi_n(q)$.

\subsection{Graph-theoretic results}
Let $G$ be a permutation group on a finite set $V$ and let $v\in V$.
The stabiliser of $v$ in $G$ is denoted by $G_v$, and the $G$-orbit containing $v$ is denoted by $v^G$.
Suppose that $G$ is transitive on $V$ and let $w\in V$.
A \(G_v\)-orbit \(w^{G_{v}}\) is called a \emph{\(G\)-suborbit} relative to \(v\), and it is said to be \emph{self-paired} if there exists \(h\in G\) such that \((w,v)^h=(v,w)\) (if no such \(h\) exists, then \(w^{G_{v}}\) is said to be \emph{non-self-paired}). 
If \(w^{G_{v}}\) is non-self-paired, then the digraph with vertex set \(V\) and arc set \((v, w)^G\) is a \(G\)-arc-transitive digraph. 
Conversely,  any $G$-arc-transitive digraph arises in this way. 
Note tha the $G_v$-orbits in $V$ are in one-to-one correspondence with the double cosets $G_vgG_v$ as $g$ runs over $G$.
In particular, if $w^{G_{v}}$ is non-self-paired, then it corresponds to $G_vgG_v$ with $g^{-1}\notin G_vgG_v$, where $g$ is any element such that $v^g=w$.

Let $\Gamma$ be a digraph. We use $V(\Gamma)$ and $\mathrm{Aut}(\Gamma)$ to denote the vertex set and the full automorphism group of $\Gamma $, respectively.
Let  $(v_0,v_1, \dots, v_s)$  be an $s$-arc of $\Gamma$ and $G\leqslant \mathrm{Aut}(\Gamma)$.
We use $G_{v_0v_1\cdots v_s}$ to denote the stabiliser of the $s$-arc; in particular, $G_{v_0v_1\cdots v_s}= G_{v_0}\cap G_{v_1} \cap \cdots \cap G_{v_s}$.
Recall that a \emph{factorisation} of a group $G$ is an expression $G = AB$ with $A, B \leqslant G$; if in addition $A \cong B$, then $G = AB$ is called a \emph{homogeneous factorisation}.

\begin{lemma}[{\cite[Lemma 2.2]{quasi}}] \label{pro:homofac}
Let $\mathit{\Gamma}$ be an $G$-arc-transitive digraph, and let $ u\to  v\to w $ be a $2$-arc of $\mathit{\Gamma}$.
Then  $\mathit{\Gamma}$ is $(G,2)$-arc-transitive if and only if  $G_v=G_{uv}G_{vw}$. 
\end{lemma}

\begin{lemma}[{\cite[Corollary 2.11]{quasi}}]\label{pro:HsMs-1}
 Let $\Gamma$ be an $(G,s)$-arc-transitive digraph with $s\geqslant 2$, and let $L$ be a vertex-transitive normal subgroup of $G$. 
Then  $\Gamma$ is $(L,s-1)$-arc-transitive.
\end{lemma}

The next result plays an important role in our proof of Theorem~\ref{mainthm}.

\begin{lemma}{\rm \cite[Lemma 2.14]{linear}\label{factor}}
Let \(\Gamma\) be a connected $G$-vertex-primitive \(G\)-arc-transitive digraph with arc \(v\rightarrow w\). Let \(g\in G\) such that \(v^{g}=w\). Then \(g\) normalises no proper nontrivial normal subgroup of \(G_{v}\).
\end{lemma}

\begin{lemma}[{\cite[Lemma~2.13]{linear}}]\label{pro:valency}
For each vertex-primitive arc-transitive digraph $\mathit{\Gamma}$, either $\mathit{\Gamma}$ is a directed cycle of prime length
or $\mathit{\Gamma}$ has valency at least $3$.
\end{lemma}

The following result can be derived from~\cite[Lemma 2.11]{ex1}. Recall that $\mathrm{Out}(L)$ is the outer automorphism group of a group $L$.

\begin{lemma}\label{3ATprime}
Let  $\Gamma$ be a connected $(G, s)$-arc-transitive digraph with $s\geqslant 2$, and let $u\to v \to w$ be a $2$-arc of $\Gamma$.
Suppose that \(G\) is an almost simple group with socle \(L\) and $s\geqslant 3$.  Then $|L_v|_r^2\leqslant|L_{uv}|_r^3|\Out(L)|_r$ for each prime $r$.
\end{lemma}

Now we state some results about vertex stabilisers of $2$-arc-transitive digraphs.

\begin{lemma}[{\cite[Lemma 2.13]{ex1}}]\label{lm:m2.o}
Let  $\Gamma$ be a connected $(G, s)$-arc-transitive digraph with $s\geqslant 2$, and let $u\to v \to w$ be a $2$-arc of $\Gamma$.
Suppose that $L$ is a vertex-transitive normal subgroup of $G$ and  $L_{v}=\mathrm{C}_{m}^2.\mathcal{O}$.
If there exists prime divisor $r$ of $m$ with $\vert \mathcal{O}\vert_{r}=1$ and $\vert G/L\vert_{r}<m_{r}$, then $s\leqslant 2$. 
\end{lemma}

\begin{lemma}[{\cite[Lemma 2.14]{ex1}}] 
\label{lm:m.o}
Let  $\Gamma$ be a connected $(G, s)$-arc-transitive digraph with $s\geqslant 1$, and let $u\to v \to w$ be a $2$-arc of $\Gamma$.
Suppose that $L$ is a vertex-transitive normal subgroup of $G$ and  $L_{v}=\mathrm{C}_{m}.\mathcal{O}$.
If there exists a prime divisor $r$ of $m$ with $\vert \mathcal{O}\vert_{r}=1$ and $\vert G/L\vert_{r}<m_{r}$, then $s\leqslant 1$. 
\end{lemma}

Recall that a group $G$ is called \emph{quasisimple} if $G=G'$ and $G/\mathrm{Z}(G)$ is isomorphic to a finite non-abelian simple group.

\begin{lemma}[{\cite[Lemma~3.3]{small}}]\label{lm:qsimple}
Let $\mathit{\Gamma}$ be a connected $(G, s)$-arc-transitive digraph with $s\geqslant 2$, and let $u\to v \to w$ be a $2$-arc of $\mathit{\Gamma}$.  
Suppose that  $G_v^{(\infty)}$ is quasisimple.
Let $S$ be the socle  of $G_v/\Rad(G_v)$. Then the following statements hold:
\begin{enumerate}[\rm (a)]   
\item  The simple group  $S$ is  isomorphic to one of the following:
\begin{equation}\label{eq:groupsqs}
\mathrm{A}_6,\mathrm{M}_{12},\mathrm{Sp}_4(2^t),\mathrm{P\Omega}^{+}_8(q), \mathrm{PSL}_2(q),\mathrm{PSL}_3(3),\mathrm{PSL}_3(4),\mathrm{PSL}_3(8),\mathrm{PSU}_3(8),\mathrm{PSU}_4(2). 
\end{equation}  
Moreover, if $G_v$ is almost simple, then $S$ is isomorphic to $\mathrm{A}_6$, $\mathrm{M}_{12}$, $\mathrm{Sp}_4(2^t)$ or $\mathrm{P\Omega}^{+}_8(q)$ and $s\leqslant 2$.

\item Suppose that $S $ is one in~\eqref{eq:groupsqs}. 
Let  
\[ \overline{G_v}=G_v/\mathrm{R}(G_v),\,\overline{G_{uv}}=G_{uv} \mathrm{R}(G_{ v})/\mathrm{R}(G_{ v}),\, \overline{G_{vw}}= \overline{G_{vw}} \mathrm{R}(G_{ v})/\mathrm{R}(G_{v}).\]
Then both $\overline{G_{uv}}$ and $\overline{G_{vw}}$ are core-free in $\overline{G_{v}}$, and the factorisation $\overline{G_{v}}=\overline{G_{uv}}\,\overline{G_{vw}}$ satisfies
one of the following (interchanging $\overline{G_{uv}}$ and $\overline{G_{vw}}$ if necessary):
\begin{enumerate}[\rm (b.1)] 
\item $S=\mathrm{Sp}_4(2^f)$ with $f\geqslant 2$, $\overline{G_{v}}\leqslant\mathrm{\Gamma Sp}_4(q)$, and $(\overline{G_{uv}}\cap S,\overline{G_{vv_1}}\cap S)\cong(\mathrm{Sp}_2(q^2).2,\mathrm{Sp}_2(q^2))$ or $(\mathrm{Sp}_2(q^2).2,\mathrm{Sp}_2(q^2).2)$; 
\item $S=\mathrm{P\Omega}_8^{+}(q)$, $\overline{G_{v}}\leqslant  \mathrm{P\Gamma O}_8^{+}(q) $, and $(\overline{G_{uv}}\cap S,\overline{G_{vw}}\cap S)\cong (\Omega_7(q),\Omega_7(q))$.
 \item $S=\mathrm{PSL}_2(q)$, $\overline{G_{uv}}\cap S\leqslant \mathrm{D}_{2(q+1)/d}$ and $ \overline{G_{vw}}\cap S \leqslant q{:}((q-1)/d)$, where $d=(2,q-1)$.
\item $(\overline{G_{v}},\overline{G_{uv}},\overline{G_{vw}})$ satisfies Table \ref{tb:exceptionalfacs}, where $\mathcal{O}\leqslant \mathrm{C}_2$, and $\mathcal{O}_1$ and $\mathcal{O}_2$ are subgroups of $\mathcal{O}$ such that $\mathcal{O}=\mathcal{O}_1\mathcal{O}_2$.

\end{enumerate} 

\end{enumerate} 
\end{lemma}

\begin{table}[ht]
\centering
\caption{Some factorisations for $\overline{G_v} = \overline{G_{uv}} \overline{G_{vw}}$}
\label{tb:exceptionalfacs}
\begin{tabular}{lll}
\hline 
\noalign{\vspace{0.5ex}}
$\overline{G_v}$ & $\overline{G_{uv}}$ & $\overline{G_{vw}}$ \\
\hline
$\mathrm{A}_6.\mathcal{O} (\leqslant \mathrm{S}_6)$ & $\mathrm{A}_5.\mathcal{O}$ & $\mathrm{A}_5.\mathcal{O}$ \\
$\mathrm{M}_{12}$ & $\mathrm{M}_{11}$ & $\mathrm{M}_{11}$ \\
$\mathrm{PSL}_2(7).\mathcal{O}$ & $7, 7:(3 \times \mathcal{O})$ & $\mathrm{S}_4$ \\
$\mathrm{PSL}_2(11).\mathcal{O}$ & $11, 11:(5 \times \mathcal{O})$ & $\mathrm{A}_4$ \\
$\mathrm{PSL}_2(23).\mathcal{O}$ & $23:(3 \times \mathcal{O})$ & $\mathrm{S}_4$ \\
$\mathrm{PSL}_3(3).\mathcal{O}$ & $13, 13:(3 \times \mathcal{O})$ & $3^2.2.\mathrm{S}_4$ \\
$\mathrm{PSL}_3(3).\mathcal{O}$ & $13:(3 \times \mathcal{O})$ & $\mathrm{A\Gamma L}_1(9)$ \\
$\mathrm{PSL}_3(4).(\mathrm{S}_3 \times \mathcal{O})$ & $7:(3 \times \mathcal{O}).\mathrm{S}_3$ & $2^4:(3 \times \mathrm{D}_{10}).2$ \\
$\mathrm{PSL}_3(8).(3 \times \mathcal{O})$ & $57:(9 \times \mathcal{O}_1)$ & $2^{3+6}:7^2.(\mathrm{D}_{10} \times \mathcal{O}_2)$ \\
$\mathrm{PSU}_3(3).\mathcal{O}$ & $57:(9 \times \mathcal{O}_1)$ & $2^{3+6}:(63:3).\mathcal{O}_2$ \\
$\mathrm{PSU}_4(2).\mathcal{O}$ & $2^4.5$ & $3^{1+2}.2:(\mathrm{A}_4.\mathcal{O})$ \\
$\mathrm{PSU}_4(2).\mathcal{O}$ & $2^4.5$ & $3^{1+2}.2:(\mathrm{A}_4.\mathcal{O})$ \\
$\mathrm{PSU}_4(2).\mathcal{O}$ & $2^4:\mathrm{D}_{10}.\mathcal{O}_1$ & $3^{1+2}.2:(\mathrm{A}_4.\mathcal{O}_2)$ \\
$\mathrm{PSU}_4(2).2$ & $2^4.5:4$ & $3^{1+2}:\mathrm{S}_3, 3^3:(\mathrm{S}_3 \times \mathcal{O}),$ \\
& & $3^3:(\mathrm{A}_4 \times 2), 3^3:(\mathrm{S}_4 \times \mathcal{O})$ \\
$\mathrm{M}_{11}$ & $11:5$ & $\mathrm{M}_9.2$ \\
\hline
\end{tabular}
\end{table}

\section{Proof of Theorem \ref{mainthm}}
Throughout this section, we work under the following hypothesis.

\begin{hypothesis}\label{hy:1}
Let $G$ be an almost simple group with socle $L \in \{ E_7(q), E_8(q) \}$, where $q = p^f$ with $p$ a prime, and let $\Gamma$ be a connected $G$-vertex-primitive $(G, s)$-arc-transitive digraph with $s \geqslant 2$. 
Fix a $2$-arc $u \rightarrow v \rightarrow w$ of $\Gamma$ and let $g \in L$ be such that $(u,v)^g = (v,w)$ (this is, $u^g=v$ and $v^g=w$).  
\end{hypothesis}

Note that the existence of such an element $g \in L$ satisfying $u^g = v$ and $v^g = w$  is guaranteed by Proposition~\ref{pro:HsMs-1}.

Since \(\Gamma\) is \(G\)-vertex-primitive,  all vertex stabilisers in $G$ are conjugate and the vertex-stabiliser $G_v$ is a core-free maximal subgroup of \(G\).
 
We will examine all conjugacy representatives of the core-free maximal subgroups of $G$, which  are classified in \cite{alex, LS}.
According to \cite[Theorem 2]{LS}, either $G_v$ is almost simple, or one of the following holds:

\begin{enumerate}[\rm (a)] 
\item $G_v$ is a parabolic subgroup, or a subgroup of maximal rank.

\item $G_v $ is the normaliser of an elementary abelian group given in~\cite[Theorem 1(II)]{CLSS}.

\item $G_v $ is the centraliser of a graph, field, or graph-field automorphism of $L$ of prime order. 
\item $L=E_8(q)$, $p>5$ and $\mathrm{F}^{*}(G_v )$ is either $\mathrm{A}_5 \times \mathrm{A}_6$ or  $\mathrm{A}_5 \times \mathrm{PSL}_2(q)$.
\item $(L,\mathrm{F}^{*}(G_v ))$ is one of the following:
\begin{itemize}
\item $L=E_7(q)$, and $\mathrm{F}^{*}(G_v)=\mathrm{PSL}_2(q)\times \mathrm{PSL}_2(q)$ ($p>3$), 
$\mathrm{PSL}_2(q)\times G_2(q) $ ($p>2$, $q>3$), 
$\mathrm{PSL}_2(q)\times F_4(q) $ ($q>3$),
$\mathrm{PSp}_6(q)\times G_2(q) $ ($q>3$).
\item $L=E_8(q)$, and $\mathrm{F}^{*}(G_v)= \mathrm{PSL}_2(q)\times \mathrm{PSL}_3(q) $ ($p>3$), 
$\mathrm{PSL}_2(q)\times \mathrm{PSU}_3(q) $ ($p>3$),  
$G_2(q)\times F_4(q) $,
$\mathrm{PSL}_2(q)\times G_2(q)^2 $ ($p>2$, $q>3$),
$\mathrm{PSL}_2(q)\times G_2(q^2) $ ($p>2$, $q>3$).
\end{itemize} 
\end{enumerate}

All maximal subgroups appearing in (a)--(e) have been determined up to conjugacy.
In the parabolic case, the group structure can be read directly from the corresponding Dynkin diagram.
The subgroups of maximal rank are recorded
in \cite[Tables 5.1 and 5.2]{LSS}.
The subgroups in (b), also known as exotic $r$-local subgroups (for a prime $r\neq p$), satisfy, by \cite[Theorem 1(II)]{CLSS}, that $L=E_8(q)$ and $L_v=5^3.\mathrm{SL}_3(5)$ or $(2^5.2^{10}).\mathrm{SL}_5(2)$. 
Since  neither $E_7(q)$ nor $E_8(q)$  admits graph or graph-field automorphisms, the subgroups in (c) are of the same type as $G$, i.e., $G_v=E_7(q^{1/r})$ or $E_8(q^{1/r})$ for some prime $r$. 

For those maximal subgroups that are almost simple and are not covered by cases (a)--(e), see~\cite[Table 1.1]{Craven} for  $E_7(q)$ and see \cite[Remark 2.16]{BK2025} for a detailed overview for $E_8(q)$. 
For convenience, we let $\mathcal{S}$ be the set of those almost simple maximal subgroups not covered by cases (a)--(e).
 
Note that
\[
|\mathrm{Out}(L)|=
\begin{cases}
(2, q-1)t, &\text{if } L = E_7(q),\\[2mm]
 t, &\text{if } L = E_8(q).\\ 
\end{cases}
\]

The group $L$ possesses a $(B,N)$-pair (see for example \cite[Section 2.5]{Carter2}); in particular, $L = BNB$, where $B$ is a Borel subgroup. 
Let $W = N/(B \cap N)$  denote the Weyl group of $L$, let $\Phi$ be the set of roots corresponding to $W$, and let $\Pi = \{\alpha_1, \ldots, \alpha_\ell\}$ (with $\ell = 7$ or $8$) be a set of simple roots.
For each root \(r \in \Phi\), denote by \(X_r\) the associated root subgroup.
Then \(B = U : (B \cap N)\), where \(U = \langle X_r \mid r \in \Phi^+ \rangle\) (with \(\Phi^+\) the set of positive roots).  
The Dynkin diagram of \(L\) is shown in Figure~\ref{fg:Dynkindigrams}. 

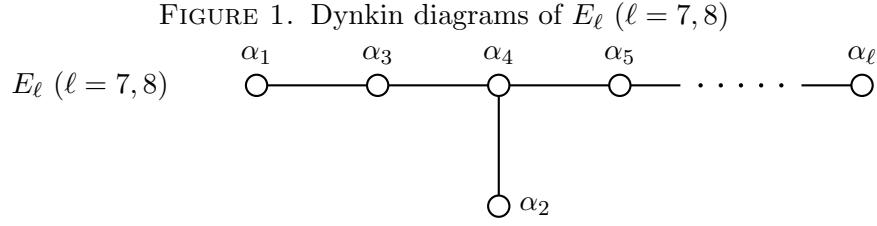
\begin{figure}[h]
\centering
\caption{Dynkin diagrams of $E_\ell$ ($\ell = 7,8$)}\label{fg:Dynkindigrams}
\begin{tikzpicture}[scale=0.4]
  \draw [thick] (2,-2.5) -- (-2,-2.5);
  \draw [thick] (2,-2.5) -- (2,-6.5); 
  \draw [thick] (-2,-2.5) -- (-6,-2.5);
  \draw [thick] (2,-2.5) -- (6,-2.5);
  \draw [thick] (6,-2.5) -- (8,-2.5);
  \draw [thick] (12,-2.5) -- (14,-2.5);

  \filldraw[thick,fill=white] (2,-2.5) circle (10pt);
  \node at (2,-1.5) {\normalsize $\alpha_4$};
  \filldraw[thick,fill=white] (-2,-2.5) circle (10pt);
  \node at (-2,-1.5) {\normalsize $\alpha_3$};
  \filldraw[thick,fill=white] (-6,-2.5) circle (10pt);
  \node at (-6,-1.5) {\normalsize $\alpha_1$};
  \filldraw[thick,fill=white] (6,-2.5) circle (10pt);
  \node at (6,-1.5) {\normalsize $\alpha_5$};
  \filldraw[thick,fill=white] (2,-6.5) circle (10pt);
  \node at (3.2,-6.5) {\normalsize $\alpha_2$};
  \filldraw[thick,fill=white] (14,-2.5) circle (10pt);
  \node at (14,-1.5) {\normalsize $\alpha_\ell$};
  \filldraw[thick,fill=white] (8.667,-2.5) circle (1pt);
  \filldraw[thick,fill=white] (9.334,-2.5) circle (1pt);
  \filldraw[thick,fill=white] (10,-2.5) circle (1pt);
  \filldraw[thick,fill=white] (10.667,-2.5) circle (1pt);
  \filldraw[thick,fill=white] (11.334,-2.5) circle (1pt);
 
  \node at (-11.5,-2.5) {\normalsize $E_\ell$ ($\ell=7,8$)};
\end{tikzpicture}
\end{figure}

In Subsection~\ref{subsec:parabolic}, we consider parabolic subgroups. 
The remaining two subsections deal with non-parabolic subgroups of $E_7(q)$ and $E_8(q)$, respectively.

\subsection{Parabolic maximal subgroups}\label{subsec:parabolic}
In this subsection, we show that, under Hypothesis~\ref{hy:1}, the case where \(G_v\) is a maximal parabolic subgroup cannot occur.
A similar result was established for \(L \in \{{}^3D_4(q), G_2(q), {}^2F_4(q), F_4(q), E_6(q), {}^2E_6(q)\}\) in \cite[Subsection 3.1]{ex1}.  
Here, however, we present a different and more efficient approach.

Suppose that Hypothesis~\ref{hy:1} holds and that \(G_v\) is a maximal parabolic subgroup of \(G\).
Then \(L_v\) is conjugate to a standard maximal parabolic subgroup \(P_J\) of \(L\) for some \(J \subseteq \Pi\) with \(|J| = |\Pi| - 1\).  
Let \(W_J\) denote the standard parabolic subgroup of \(W\) associated with \(J\), i.e. the subgroup generated by the reflections corresponding to the simple roots in \(J\).
By \cite[Propositions 2.7.5 and 2.8.1]{Carter2}, there is a bijection between \((W_J, W_J)\)-double cosets in \(W\) and \((P_J, P_J)\)-double cosets in \(L\): $W_J x W_J \leftrightarrow P_J \dot{x} P_J$, where \(x\) is the unique element of minimal length in \(W_J x W_J\)  (see \cite[Proposition 2.8.1]{Carter2}) and \(\dot{x}\) is a preimage of \(x\) in \(N\). Let \(D_{J,J}\) be the set of all such minimal-length representatives, and let \(\Phi_J\) be the set of roots that are integral combinations of roots in \(J\).

By~\cite[Theorem 2.6.5]{CFSG}, \(P_{J}=U_{J}:L_{J}\) where \(U_{J}=\prod_{r\in\Phi^{+}\cap\Phi_{J}}X_{r}\) and \(L_{J}=M_J(B\cap N)\) where $M_J=\langle X_{r}: r\in \Phi_{J}\rangle $ and is a central product of certain groups of Lie type. In the simply connected version of \(L\), the \(M_J\)-factor is a direct product of the simply connected versions of those Lie-type groups. For instance, if \(L = E_7(q)\) and \(J = \Pi \setminus \{\alpha_4\}\), then \(J\) has type \(A_1 A_2 A_3\), and hence the \(M_J\)-factor in the parabolic subgroup corresponding to \(J\) inside \({E_7(q)}_{\mathrm{sc}}\) is \(
{A_1(q)}_{\mathrm{sc}} \times {A_2(q)}_{\mathrm{sc}} \times {A_3(q)}_{\mathrm{sc}}\cong \mathrm{SL}_2(q) \times \mathrm{SL}_3(q) \times \mathrm{SL}_4(q)\).
Consequently, the \(M_J\)-factor in \(E_7(q)\) is the quotient of \( \mathrm{SL}_2(q) \times\mathrm{SL}_3(q) \times \mathrm{SL}_4(q)\) by its intersection with the centre of \({E_7(q)}_{\mathrm{sc}}\).

Since $\Gamma$ is an $L$-arc-transitive digraph, the double coset $L_v g L_v$ satisfies $g^{-1} \notin L_v g L_v$.
Let $x\in D_{J,J}$ be the corresponding Weyl element such that $L_v g L_v=P_J\dot{x}P_J$.
Because of this uniqueness, for any $y \in D_{J,J}$, we have $y^{-1} \in W_J y W_J$ if and only if $y$ is the identity or an involution.  
Therefore, the condition $g^{-1} \notin P_J \dot{x} P_J$ forces $|x|>2$. We note that only some subsets $J$ yields such Weyl element $x$ with $|x|>2$, see for example \cite[Proposition 10.4.11]{DRG}; in particular, one of the following occurs:
  \begin{enumerate}
      \item[(i)] \(L=E_7(q)\), \(J=\Pi\setminus\{\alpha_i\}\) for \(i=3,4,5,6\);
      \item[(ii)] \(L=E_8(q)\), \(J=\Pi\setminus\{\alpha_i\}\) for \(i=2,3,4,5,6,7\).
  \end{enumerate}

The group structure of \(P_J \cap P_J^{\dot{x}}\) is determined in \cite[Section 2.8]{Carter2}.
According to \cite[Theorem 2.7.4]{Carter2}, we have \(W_J \cap W_J^x = W_K\) with \(K = J \cap J^{x^{-1}}\).
Moreover, by \cite[Theorem 2.8.7]{Carter2}, \(P_J \cap P_J^{\dot{x}} = Q : L_K\), where \(Q\) is a normal \(p\)-subgroup and \(L_K = M_K (B \cap N)\).

For a given \(J \subset \Pi\), the set \(D_{J,J}\) can be  easily computed using {\sc Magma}~\cite{magma}.
For each \(x \in D_{J,J}\) with \(|x| > 2\), we compute \(K = J \cap J^{x^{-1}}\), from which we read off the type of \(K\) and thus obtain the structure information of \(M_K\).
The following sample {\sc Magma} code computes \(D_{J,J}\) and the types of \(K\) for all such \(x\):
 
\begin{verbatim}
W0:=CoxeterGroup(GrpFPCox, "E7"); J:={1..7} diff {4};
W,phi := CoxeterGroup(GrpPermCox, W0); R:=RootSystem(W); 
DJJ,_:=Transversal(W0, J, J); 
phiDJJ:=[phi(x): x in DJJ | Order(phi(x)) gt 2]; 
print "number of non-self-paired orbits:", #phiDJJ;  
list:=[]; Re:={}; 
for idw in [1..#phiDJJ] do
  w:=phiDJJ[idw]; K:=J meet J^(w^(-1)); 
  RK:=sub<R|K>; RK_name:=CartanName(RK);
  Append(~list,<K,RK_name>); 
  if not exists{i: i in Re | IsIsomorphic(i,RK)} then
	   Include(~Re,RK);
  end if; end for; 
[CartanName(i): i in Re]; list;  
\end{verbatim}
The computation shows that \(D_{J,J}\) contains \(92\) elements \(x\) with \(|x| > 2\). Across these \(92\) elements, the type of \(K\) is one of the following:
\[
A_1,\; A_1A_1,\; A_1A_1A_1,\; A_1A_1A_1A_1,\; A_1A_2,\; A_1A_1A_2.
\]

Table~\ref{tab:typeofK} lists, for a fixed subset \(J \subset \Pi\) under consideration,  all possible Dynkin types of \(K = J \cap J^{x^{-1}}\) arising from elements \(x \in D_{J,J}\) with \(|x| > 2\). 
For each row, we also record the type of $J$, the number of such elements \(x\), and a prime \(r\) dividing \(|M_J|\) but not \(|M_K|\). 
Such primes \(r\) are obtained directly by inspecting the order formulae for groups of Lie type and applying Lemma~\ref{existppd} on primitive prime divisors.

\begin{table}[h]
\centering
\caption{Types of \(K = J \cap J^{x^{-1}}\) for all $x\in D_{J,J}$ with $|x|>2$, and primes \(r\) dividing \(|M_J|\) but not \(|M_K|\)}\label{tab:typeofK}
\begin{tabular}{lllllll}

\toprule
 \(L\)  & \(J\) & number of \(x\) & type of \(J\) & type of \(K\) & prime \(r\) \\
\midrule
 \(E_7\) & \(\Pi\setminus\{\alpha_3\}\) & 8 & \(A_1A_5\) & 
\(A_2A_1\), \(A_3A_1\) & 
\(r\in\ppd(p,5f)\) \\
 \(E_7\) & \(\Pi\setminus\{\alpha_4\}\) & 92 & \(A_1A_2A_3\) & 
\(A_2A_1A_1\), \(A_1A_1A_1A_1\), \(A_2A_1\),  & 
\(r\in\ppd(p,4f)\) \\
  &  &   &  &  \(A_1A_1A_1\), \(A_1A_1\), \(A_1\) &  \\
 \(E_7\) & \(\Pi\setminus\{\alpha_5\}\) & 20 & \(A_2A_4\) & 
\(A_2A_1A_1\), \(A_2A_1\), \(A_1A_1A_1\), \(A_1A_1\) & 
\(r\in\ppd(p,5f)\) \\
 \(E_7\) & \(\Pi\setminus\{\alpha_6\}\) & 2 & \(A_1D_5\) & 
\(A_3A_1\) & 
\(r\in\ppd(p,8f)\) \\

 \(E_8\) & \(\Pi\setminus\{\alpha_2\}\) & 10 & \(A_7\) & 
\(A_4A_1\), \(A_3A_2\), \(A_2A_2\) & 
\(r\in\ppd(p,8f)\) \\
 \(E_8\) & \(\Pi\setminus\{\alpha_3\}\) & 76 & \(A_1A_6\) & 
\(A_4A_1\), \(A_4\), \(A_3A_1A_1\), & 
\(r\in\ppd(p,7f)\) \\
&&&& \(A_3A_1\), \(A_2A_2\), \(A_2A_1A_1\),& \\
&&&&\(A_2A_1\), \(A_1A_1A_1A_1\), \(A_1A_1A_1\) & 
 \\
 \(E_8\) & \(\Pi\setminus\{\alpha_4\}\) & 1208 & \(A_4A_2A_1\) & 
\(A_3A_1A_1\), \(A_3A_1\), \(A_3\), \(A_2A_2A_1A_1\), & 
\(r\in\ppd(p,5f)\) \\
&&&& \(A_2A_2A_1\), \(A_2A_2\), \(A_2A_1A_1A_1\), & \\
&&&&  \(A_2A_1A_1\), \(A_2A_1\), \(A_2\), \(A_1A_1A_1A_1\),  &\\&&&& \(A_1A_1A_1\), \(A_1A_1\), \(A_1\) & 
  \\
 \(E_8\) & \(\Pi\setminus\{\alpha_5\}\) & 412 & \(A_4A_3\) & 
\(A_3A_3\), \(A_3A_2\), \(A_3A_1A_1\),  & \(r\in\ppd(p,5f)\) \\
&&&& \(A_3A_1\), \(A_2A_2\), \(A_2A_2A_1\), & \\
&&&& \(A_2A_1A_1\), \(A_2A_1\), \(A_1A_1A_1A_1\), & \\
&&&&  \(A_1A_1A_1\), \(A_1A_1\), \(A_1\) & 
\\
 \(E_8\) & \(\Pi\setminus\{\alpha_6\}\) & 68 & \(D_5A_2\) & 
\(D_4A_1\), \(A_4A_1\), \(A_3A_1\), & \(r\in\ppd(p,8f)\)\\
&&&&  \(A_3\), \(A_2A_1A_1A_1\),\(A_2A_1A_1\), & \\
&&&&  \(A_2A_1\), \(A_1A_1A_1A_1\), \(A_1A_1A_1\) & 
 \\
 \(E_8\) & \(\Pi\setminus\{\alpha_7\}\) & 8 &\(E_6A_1\) & 
\(D_5\), \(A_4A_1\) & 
\(r\in\ppd(p,12f)\) \\
\bottomrule
\end{tabular}
\end{table}

\begin{lemma}\label{lm:para}
Suppose that Hypothesis~\ref{hy:1} holds. Then \(G_v\) is not a maximal parabolic subgroup of \(G\).
\end{lemma}

\begin{proof}
    Since $\Gamma$ is $(G,2)$-arc-transitive, we have the homogeneous factorization
\( 
G_v=G_{uv}G_{vw}
\). 
Hence \(|G_v|\) divides \(|G_{vw}|^2\).
For \(r\) recorded in Table \ref{tab:typeofK}, since \(r \mid |M_J|\) and therefore \(r \mid |G_v|\), it follows that \(|G_{vw}|\) is divisible by \(r\).
Observe that
\[
G_{vw}/L_{vw} = G_{vw}/(G_{vw} \cap L) \cong G_{vw}L/L \leqslant G/L \leqslant \mathrm{Out}(L),
\]
so $|G_{vw}|_r \leqslant |L_{vw}|_r |\mathrm{Out}(L)|$.
By Table~\ref{tab:typeofK} and Lemma~\ref{sizeppd}, \(r > 4f\).
From \(|\mathrm{Out}(L)| \mid 2f\) we have \(|\mathrm{Out}(L)|_r = 1\), which forces \(|L_{vw}|_r > 1\).
However, this leads to a contradiction since $|L_{vw}|=|Q|\cdot |L_K|$, where $|Q|$ is a $p$-power, and $|L_K|=|M_K(B\cap N)|$ with $|M_K|_r=1$ and $|B\cap N|_r=(q-1)_r=1$.

We therefore have proved the following result.
\end{proof}

\subsection{Subgroups of maximal rank}
In this subsection, we will address the maximal subgroups of maximal rank of \(E_7(q)\) and \(E_8(q)\). 

\subsubsection{\texorpdfstring{$L=E_7(q)$}{L=E7(q)}}
Here, we will prove the following result.
\begin{proposition}\label{max rk 7}
    Suppose that Hypothesis \ref{hy:1} holds and \(L=E_7(q)\). Suppose that \(L_v\) is a maximal subgroup of maximal rank. Then \(s\leqslant2\).
\end{proposition}

By \cite[Table 4.1]{Craven}, \(L_{v}\) satisfies one of the cases in Table~\ref{tab:E7maxrk}.

\begin{table}[h]
    \caption{Maximal subgroups of maximal rank of \(E_7(q)\). Notation: $d=(2,q-1)$, $\epsilon$ runs over $\{-1,1\}$, and $e_\epsilon=(3,q-\epsilon)$.}
    \label{tab:E7maxrk}
    \centering
    \begin{tabular}{@{}lll@{}}
\toprule
      Case& \(L_v\)& Remarks \\
      \midrule
      (1)&\(d.(\PSL_{2}(q)\times\POmega^{+}_{12}(q)).d\)& \\
      (2)&\(e_{\epsilon}.(\PSL_{3}^{\epsilon}(q)\times\PSL_6^{\epsilon}(q)).e_{\epsilon}.2\)&\\
      (3)&\(\PSL_{2}(q^{3})\times {}^{3}\!D_{4}(q)).3\)&\\
      (4)&\(d^{2}.(\PSL_{2}(q)^{3}\times\POmega^{+}_{8}(q)).d^{2}.\Sy_{3}\)&\\
      (5)& $\frac{(4,q-\epsilon)}{d}.\PSL_{8}^{\epsilon}(q).\frac{(8,q-\epsilon)}{(4,q-\epsilon)}.2$ &\\
      (6)& \(e_{\epsilon}.(E^{\epsilon}_{6}(q)\times (q-1)/de_{\epsilon}).e_{\epsilon}.2\)&\\
      (7)& \(\PSL_{2}(q^7).7\)&\\
      (8)&\(d^{3}.\PSL_{2}(q)^{7}.d^{3}.\PSL_{3}(2)\)&\(q\geqslant3\)\\
      (9)&\((q-\epsilon)^{7}/d.W(E_{7})\)&\(q\geqslant5\) if \(\epsilon=+1\)\\

      \bottomrule
    \end{tabular} 
\end{table}
\begin{lemma}\label{lm:e7cases1-7}
    Suppose that Hypothesis \ref{hy:1} holds and \(L=E_{7}(q)\). Then cases (1)--(4) in Table \ref{tab:E7maxrk} do not occur.
\end{lemma}
\begin{proof}
 Suppose, for a contradiction, that one of the cases (1)--(4) occurs. If \(q=2\) and one of cases (2) with \(\epsilon=+1\) or (4) occurs, then \(G=L\) and \(G_v=L_v\cong (\PSL_3(2)\times\PSL_6(2)).2\) or \((\PSL_2(2)^3\times\POmega_8^+(2)).\Sy_3\), respectively.  We check by \textsc{Magma} \cite{magma} that in these cases there is no homogeneous factorisation for \(L_v=L_{uv}L_{vw}\), which is a contradiction. Thus, we may assume that if case (2) with \(\epsilon=+1\) or (4) occurs, then \(q\geqslant3\).
 
 Note that \(\ppd(p,fi)\neq\varnothing\) for \(i\in\{3,4,6,10,12\}\). For each of such \(i\), choose \(r_i\in\ppd(p,fi)\). Note that \(G_v\) has a normal subgroup \(M\) such that \(J\subset\pi(M)\). Let \(C=\C_{G_v}(M)\) and \(\overline{H}:=HC/C\) for any subgroup \(H\leqslant G_{v}\). Then \(\overline{G_v}\) is almost simple with socle \(T\) where \((M,J,\N_{\Aut(L)}(C),T)\) as listed in Table \ref{tab:E7iv}.  
 \begin{table}[h]
     \centering
          \caption{Possibilities for \((M,J,\N_{\Aut(L)}(C),T)\) for cases (1)--(5)}
     \label{tab:E7iv}
     \begin{tabular}{@{}lllll@{}}
\toprule
       Case   &M&J&\(\N_{\Aut(L)}(C)\)&\(T\)  \\
       \midrule
       (1) & \(\Omega_{12}^+(q)\)&\(\{r_{10}\}\)  & \(\SL_2(q).\Out(L)\)&\(\POmega_{12}^+(q)\)\\
       (2)& \(\SL_{6}^\epsilon(q)\)&\(\{r_{9-\epsilon}\}\)&\(\SL^\epsilon_3(q).\Out(L)\)&\(\PSL^\epsilon_{9-\epsilon}(q)\)\\
       (3)&\({}^3\!D_4(q)\)&\(\{r_{12}\}\)&\(\PSL_2(q^3).\Out(L)\)&\(\PSL_2(q^3)\)\\
       (4)&\(\Omega_8^+(q)\)&\(\{r_3,r_4,r_6\}\)&\(d^2.\PSL_2(q)^3.\Out(L)\)&\(\POmega_8^+(q)\)\\
       \bottomrule
     \end{tabular}
 \end{table}

For cases (1)--(2), note that for \(r\in J\), since \(|\Out(L)|=(2,q-1)f\) and \(r>3f\), we have \(|\Out(L)|_r=1\) and therefore \(|C|_r=|\N_{\Aut(L)}(C)|_r=1\). Thus, \(|\overline{G_{uv}}|_{r}=|\overline{G_{vw}}|_{r}>1\). Then it follows from \cite[Tables 1--3, 5]{LPS} that at least one of \(\overline{G_{uv}}\) or \(\overline{G_{vw}}\) contains \(T\). Without loss of generality, assume that \(\overline{G_{uv}}\geqslant T\). 

For cases (3), from \cite[Table 5]{LPS} it follows that \(\overline{G_{v}}\) does not admit a core-free factorisation. Thus, \({}^3D_4(q)\cong T\leqslant\overline{G_{uv}}\). 

Now, we consider case (4).  Here \(q\geqslant3\) and so \(r_3\neq r_6\). Note that \(\overline{G_{uv}}\,\overline{G_{vw}}=\overline{G_{v}}\cong \Aut(\POmega^{+}_{8}(q))\). Since  \(|\overline{G_{uv}}|_{r}=|\overline{G_{vw}}|_{r}>1\) for \(r\in J\). By \cite[Table 1]{LPS} we deduce that at least one of \(\overline{G_{uv}}\) or \(\overline{G_{vw}}\) contains \(\POmega^{+}_{8}(q)\). Without loss of generality, assume that \(\overline{G_{uv}}\geqslant T\cong \POmega^{+}_{8}(q)\). 
 
In all cases, \(T\leqslant \overline{G_{uv}}\), it then follows that \(M\trianglelefteq G_{uv}\). Since \(G_{uv}^{g}=G_{vw}\), we deduce that \(M^{g}\) is also a normal subgroup of \(G_{vw}\). However, \(G_v\) contains no other subgroup isomorphic to \(M\). Hence, \(M^{g}=M\), contradicting Lemma~\ref{factor}. Thus, the result is proved.
\end{proof}

\begin{lemma}
 Suppose that Hypothesis \ref{hy:1} holds and \(L=E_{7}(q)\). Then cases (5)--(7) in Table \ref{tab:E7maxrk} do not occur. 
\end{lemma}
\begin{proof}
 For these three cases, \(G_{v}^{(\infty)}\) is quasisimple, with insoluble composition factor $\PSL_8^{\epsilon}(q)$, $E_6^{\epsilon}(q)$ and $\mathrm{PSL}_2(q^7)$, respectively.
 According to Lemma~\ref{lm:qsimple}(a), the cases (5) and (6) can not happen.

Assume that case (7) holds.
Let $\overline{G_v}=G_v/\mathrm{R}(G_v) $,  $\overline{G_{uv}}=G_{uv} \mathrm{R}(G_{v})/\mathrm{R}(G_{ v})$ and $\overline{G_{vw}}= \overline{G_{vw}} \mathrm{R}(G_{ v})/\mathrm{R}(G_{v})$ be as in Lemma~\ref{lm:qsimple}(b).
Let $r\in \mathrm{ppd}(p,14f)$. Then $r\mid (q^7+1)$ and $r>14f$.
By Lemma~\ref{lm:qsimple}(b.4), interchanging $\overline{G_{uv}}$ and $\overline{G_{vw}}$ if necessary, we have $ |\overline{G_{uv}}|_r\geqslant r$ and $|\overline{G_{vw}}|_r=1$. 
Since $L_v=\mathrm{PSL}_2(q^7).7$ and $|\mathrm{Out}(L)|=(2,q-1)f$, we have $|R(G)|$ divides $14f$.
Hence $|R(G)|_r=1$, and consequently, $|G_{vw}|_r=1$.
This contradicts the homogeneous factorisation $G_v=G_{uv}G_{vw}$.

Hence, the result is proved.
\end{proof}

\begin{lemma}\label{E7weyl}
  Suppose that Hypothesis \ref{hy:1} holds and \(L=E_{7}(q)\). Then case (8) in Table \ref{tab:E7maxrk} does not occur.
\end{lemma}
\begin{proof}
Suppose, for a contradiction, that case (8) occurs.
Then $L_v=d^3.\mathrm{PSL}_2(q)^7.d^3.\mathrm{PSL}_3(2)$ and \(G_v=L_v.m\), where $d=(2,q-1)$, $q=p^f\geqslant 3$ and \(m\mid f\).
Now $L_v$ is a subgroup of maximal rank, and it arises from a subsystem $\Delta$ of type $ (A_1)^7$ of $E_7$. 
The subsystem $\Delta$ can be taken to be the one with simple roots $\{ \alpha_2,\alpha_3,\alpha_5,\alpha_7, \alpha_0, \beta_1,  \beta_2\}$, where 
\begin{align*}
 &\alpha_0= 2\alpha_1 +2\alpha_2 +3\alpha_3 +4\alpha_4 +3\alpha_5 +2\alpha_6 + \alpha_7,\\
 &\beta_1=  \alpha_2 +\alpha_3 +2\alpha_4 +2\alpha_5 +2\alpha_6 +\alpha_7 , \\
 &\beta_2=\alpha_2 + \alpha_3 + 2\alpha_4 + \alpha_5.
\end{align*}

Since a subsystem $(A_1)^4$ of $\Delta$ is generated by the four simple roots  $\alpha_2,\alpha_3,\alpha_5,\alpha_7$ of $E_7$, it follows from \cite[Proposition 2.6.2]{CFSG} that  $\mathrm{SL}_2(q)^4<E_7(q)_{\mathrm{sc}}=d.E_7(q)$.
Let $C\cong \mathrm{C}_d$ be the centre of $E_7(q)_{\mathrm{sc}}$ and let $D\cong \mathrm{C}_d^4$ be the centre of the group $\mathrm{SL}_2(q)^4$.
If \(q\) is odd and so \(d=2\), by computation on the Weyl group and root data in \textsc{Magma} \cite{magma}, one can verify that the group $ \mathrm{PSL}_3(2)$ acts primitively on the seven $A_1$ factors and acts irreducibly on $D/C\cong\C_2^3$.  

Let $M $ be the normal subgroup $d^3.\mathrm{PSL}_2(q)^7 $ of $H_v$ and $Z=d^3 $ the centre of $M$ (so $Z=D/C$).  
We note that $Z$ is normal in $H_v$ by analysing diagonal and field automorphisms of $L$ contained in $N_{\mathrm{Aut}(L)}(L_v)$ using~\cite[Proposition 2.6.2 and Theorem 2.5.1]{CFSG}. 
Thus, $G_v/\C_{G_v}(Z)\cong \mathrm{PSL}_3(2)$ if $Z>1$.

Denote by \(\overline{A}:=AZ/Z\) for any subgroup \(A\leqslant G_v\). Then \(\overline{G_v}\cong \PSL_2(q)^7.d^3.\PSL_3(2).m\) and \(\overline{M}=\prod_{i=1}^7\overline{M}_i\) where \(\overline{M}_1\cong\cdots\cong\overline{M}_7\cong\PSL_2(q)\). Let \(R\) be the kernel of the induced action of \(\overline{G_v}\) on \(\{\overline{M}_1,\ldots,\overline{M}_7\}\).

\medskip
\textsc{Claim 1.}  Let $A \leqslant G_v$ and suppose that $q$ is odd. If $\overline{A}\geqslant \overline{M}$ and $|\overline{A}\,R/R|_7=7$, then $A \geqslant M$.\quad

Note that \(Z>1\) as \(q\) is odd. Since $\overline{A}\geqslant \overline{M}$, it suffices to show $A \geqslant Z$. 
Suppose, for a contradiction, that $A \cap Z <Z$. 
First, we assume $A\cap Z>1$. Note that \(\overline{A}\,R/R\leqslant \overline{G_v}R/R\cong\PSL_3(2)\). Since $|\overline{A}\,R/R|_7=7$, we deduce that \(\C_7\leqslant \overline{A}\,R/R\). By \textsc{Magma} \cite{magma}
computation, we see that $\mathrm{C}_7$ is an irreducible subgroup of $\mathrm{PSL}_3(2)$ acting on $Z$. This implies that  $A\cap Z=Z$, contradicting the assumption that \(A\cap Z<Z\).
Now, $A\cap Z=1$. Since $ZA/Z=\overline{A}\geqslant \overline{M}=M/Z$, we have $  ZA \geqslant M$, and so $M=Z(A\cap M)=Z\times (A\cap M)$.
Note that $M$ is a central product of seven components $\mathrm{SL}_2(q)$ and that $\mathrm{SL}_2(q)$ has a unique involution as \(q\) is odd.
Thus, all involutions of $M$ are contained in $Z$. However, $M= Z\times (A\cap M)$ implies that $M$ contains an involution not in $Z$, a contradiction.
Therefore, Claim 1 holds.

\smallskip

Next, we divide our analysis into two cases: $q=3$ and $q>3$.

\noindent{\bf Case 1.} Suppose that $q=3$. 

Note that $\mathrm{PGL}_2(3)\cong \mathrm{S}_4$.
Now $G=L$ or $G= L.2$, and $\overline{G_v} $ is a subgroup of $\mathrm{S}_4\wr \mathrm{PSL}_3(2)$ with the index $2^4$ or $2^3$, respectively.
Computation in \textsc{Magma} \cite{magma} shows that $\mathrm{S}_4 \wr \mathrm{PSL}_3(2)$ has four classes of subgroups of indices $2^{3}$ or $2^4$ which have epimorphisms to $\mathrm{PSL}_3(2)$, and we take them as candidates for $\overline{G_v}$. 
Note that $\vert L_v\vert=2^{22}\cdot 3^8\cdot 7$.
Thus, $\vert G_{uv}\vert$ is divisible by $2^{11}\cdot 3^4\cdot 7$ when $G=L$, or by $2^{12}\cdot 3^4\cdot 7$ when $G=L.2$. 
In the factorisation $\overline{G_v}=  \overline{G_{uv}}\,\overline{G_{vw}}$, we may assume that $\vert \overline{G_{uv}} \vert_2\geqslant \vert \overline{G_{vw}} \vert_2$.
Then, two factors satisfy the following conditions:  
\begin{itemize}
\item Both $\vert \overline{G_{uv}} \vert$ and $\vert \overline{G_{vw}} \vert$ are divisible by $2^8\cdot 3^4\cdot7$.
\item $\overline{G_{uv}}/\mathbf{O}_2(\overline{G_{uv}})\cong  \overline{G_{vw}}/\mathbf{O}_2(\overline{G_{vw}})$, as $\overline{G_{uv}}\cong G_{uv}/(G_{uv}\cap Z)$, $\mathbf{O}_2(\overline{G_{uv}}) =\mathbf{O}_2(G_{uv})/(G_{uv}\cap Z)$, and $G_{uv}\cong G_{vw}$.
In particular, the Sylow $3$-subgroups of $ \overline{G_{uv}}$ and $\overline{G_{vw}} $ are isomorphic, and the Sylow $2$-subgroups of $\overline{G_{uv}}/\mathbf{O}_2(\overline{G_{uv}}) $ and $\overline{G_{vw}}/\mathbf{O}_2(\overline{G_{vw}}) $ are isomorphic.  
\item $\vert \overline{G_{uv}} \vert_2 \leqslant 2^3\cdot \vert \overline{G_{vw}} \vert_2$. 
\end{itemize}    
Computation in {\sc Magma} shows that, for all possible factorisations $\overline{G_v}= \overline{G_{uv}}\,\overline{G_{vw}}$, we find that both  $\overline{G_{uv}}$ and $\overline{G_{vw}}$ contain  $\overline{M}$ and  have epimorphisms to $\mathrm{PSL}_3(2)$. 
By Claim 1, we conclude that both $G_{uv}$ and $G_{vw}$ contain $Z$. 
It follows from $G_{uv}^g=G_{vw}$ that $Z$ is normalised by $g$, contradicting Lemma~\ref{factor}.

\smallskip
\noindent{\bf Case 2.} Suppose that $q>3$. 

 Note that $\overline{G_{v}}\leqslant \mathrm{P\Gamma L}_2(q)\wr \PSL_3(2)$.
 Recall that $ \overline{G_{v}}/R \cong \mathrm{PSL}_3(2)$.
We consider the factorisation $ \overline{G_{v}}/R=  (\overline{G_{uv}}\,R/R) (\overline{G_{vw}}\,R/R)$. Note that $|\mathrm{PSL}_3(2)|=2^3\cdot 3\cdot 7$.
Without loss of generality, we assume that $|\overline{G_{uv}}\,R/R|_7=7$, and so $\overline{G_{uv}}$ acts primitively on $\{\overline{M}_1,\ldots,\overline{M}_7\}$. 
Let $\varphi_i$ be the projection of $\overline{M}$ to $\overline{M}_i$ for each $i\in \{1,\ldots,7\}$.
Then \(\varphi_1(\overline{G_{uv}}\cap \overline{M})\cong\cdots\cong \varphi_7(\overline{G_{uv}}\cap \overline{M})\) and so 
\begin{equation}\label{eq:7A1-0}
 \vert \overline{G_{uv}} \cap \overline{M} \vert \text{  divides } \vert \varphi_1(\overline{G_{uv}}\cap \overline{M})\vert^7. 
 \end{equation}
Let $r$ be any prime. Then 
\[
\left(\frac{|\overline{G_{uv}}| }{|\overline{G_{uv}} \cap \overline{M}|}\right)_r=\left(\frac{|\overline{G_{uv}}\, \overline{M}|}{|\overline{M}|}\right)_r \leqslant  \left(\frac{|\overline{G_{v}} |}{|\overline{M}|}\right)_r \leqslant   \left(\frac{|L_{v}|\cdot |\mathrm{Out}(L)|}{|\overline{M}|}\right)_r=\left(d^7\cdot 2^3\cdot 3\cdot 7\cdot f\right)_r.
\]
Moreover, from $\vert G_{v}\vert$ dividing $\vert G_{uv}\vert^2 $, we have
\[
|\overline{G_{uv}}|_r=\frac{|G_{uv} |_r}{|G_{uv}\cap Z|_r} \geqslant \frac{|G_v|_r^{1/2}}{|Z|_r}\geqslant   |L_v|_r^{1/2} d^{-3}_r=|\mathrm{PSL}_2(q)|_r^{7/2}(2^3\cdot 3\cdot 7)^{1/2}_r.
\]
Combined with these two inequalities, we have
\begin{equation}\label{eq:7A1-1}
|\overline{G_{uv}} \cap \overline{M}|_r\geqslant \left(\frac{ |\mathrm{PSL}_2(q)| ^{7/2} }{d^7\cdot f \cdot (2^3\cdot 3\cdot 7)^{1/2}}\right)_r. 
\end{equation}
In particular, since $f_p\leqslant q^{1/p} \leqslant q^{1/2}$ (see~\cite[Lemma 2.2]{ex1}), it follows from Eq.~\eqref{eq:7A1-1} that 
\begin{equation}\label{eq:7A1-2}
|\overline{G_{uv}}|_p\geqslant 
\begin{cases}
2^{(7f-f-3)/2}\geqslant q^{\frac{3}{2}}, & \text{ if $p=2$},\\
q^{5/2}, &\text{ if $p>2$}.
\end{cases}
\end{equation}
Consequently,   $|\overline{G_{uv}}\cap \overline{M}|$ is divisible by $p$ by Eq.~\eqref{eq:7A1-0}.

 \medskip
\textsc{ Claim~2:} $\varphi_1(\overline{G_{uv}}\cap \overline{M})= \overline{M}_1$. 
 
Suppose that $q=p=2^m-1$ is a Mersenne prime. Then $m\geqslant 3$ as \(q\ne 3\). 
By Eq.~\eqref{eq:7A1-1}, $| \overline{G_{uv}}\cap \overline{M} |_2\geqslant 2^{(7m-11)/2}>1$. Thus, $2,p\in\pi( \varphi_1(\overline{G_{uv}}\cap \overline{M}))$.
According to~\cite[Table~10.3]{transitive},  $\varphi_1(\overline{G_{uv}} \cap \overline{M})=\overline{M}_1$.

Suppose that $q$ is not a Mersenne prime. If $q=4$, $8$ or $16$, then by Eq.~\eqref{eq:7A1-1} we have $\pi(\varphi_1(\overline{L_{uv}}\cap \overline{M}))=\{2,3,5\}$, $\{2,3,7\}$  or $\{2,3,5,17\}$, respectively. A \textsc{Magma} \cite{magma} computation shows that $\varphi_1(\overline{L_{uv}}\cap \overline{M})=\overline{M}_1$.
For other $q$,  we have \(\ppd(p,2t)\neq \varnothing\). Let $r \in \mathrm{ppd}(p,2t)$ and \((p^{2t}-1)_r=r^a\). Then $r>2t$, and so $| \overline{G_{uv}}\cap \overline{M}|_r\geqslant 2^{(7a-1 )/2}>1$.
Now, $r,p\in \pi( \phi_1(\overline{G_{uv}}\cap \overline{M}))$ and by~\cite[Table~10.3]{transitive},  $\varphi_1(\overline{G_{uv}} \cap \overline{M} )=\overline{M}_1$. 
Hence, Claim~2 is proved.

\medskip
 
Since $ \overline{G_{uv}} $ acts primitively on $\{\overline{M}_1,\ldots,\overline{M}_7\}$, we conclude from Scott's Lemma \cite[Theorem 4.16]{csaba} that $\overline{G_{uv}} \cap \overline{M}$ is either  $\overline{M}\cong \mathrm{PSL}_2(q)^7$, or a diagonal group isomorphic to $\mathrm{PSL}_2(q)$.
The case  $\overline{L_{uv}}\cap \overline{M} \cong \PSL_2(q)$ contradicts  Eq.~\eqref{eq:7A1-2}.
Therefore, $\overline{G_{uv}} \cap \overline{M}=\overline{M}$.
Now, $\overline{M}\leqslant \overline{G_{uv}}$. Then Claim 1 implies $M\leqslant G_{uv}$. Since $M=G_v^{(\infty)}$, it follows from $G_{uv}^g=G_{vw}$ that $M^g=M$,  which is a contradiction to Lemma \ref{factor}. 
Thus, the lemma is proved. 
\end{proof}

\begin{lemma}\label{lm:e7-qpm1}
  Suppose that Hypothesis \ref{hy:1} holds and \(L=E_{7}(q)\). Then case (9) does not occur.
\end{lemma}
\begin{proof}
Suppose, for a contradiction, that case (9) occurs. 
Now \(L_{v}=\C_m^{7}/d.W(E_{7})\), where $d=(2,q-1)$, \(m=q-\epsilon\) with \(\epsilon\in\{1,-1\}\), and $q\geqslant 5$ if $\epsilon=+1$. 
Note that  $W=W(E_7)\cong 2 \times \mathrm{PSp}_6(2)$ with $\vert W\vert= 2^{10} \cdot 3^4 \cdot 5 \cdot 7$ and  $G_v\leqslant \C_m^{7}.W.f$ (see~\cite[Table 5.2]{LSS}). Set  $M=G_{v} \cap \C_m^{7}$. 
Then we have the following useful relation:
\[ G_{uv}/(G_{uv}\cap M)\cong MG_{uv}/M \cong G_v/M \lesssim W.f.\]


Let $\overline{G_{v}}=G_v/\Rad(G_v)$. 
Then $\Soc(\overline{G_{v}})=\mathrm{PSp}_6(2)$.
Since $\mathrm{Out}(\mathrm{PSp}_6(2) )=1$, we have $\overline{G_{v}}=\mathrm{PSp}_6(2) $.
For any subgroup $X$ of $G_v$, set $\overline{X}=X \Rad(G_v)/\Rad(G_v)$.
In the factorisation $\overline{G_{v}}= \overline{G_{uv}}\, \overline{G_{vw}}$, since $G_{uv}\cong G_{vw}$, we have \(\ICF(\overline{G_{uv}})=\ICF(\overline{G_{vw}})\).
We compute all factorisations of $\mathrm{PSp}_6(2)$ in {\sc Magma}\cite{magma} where two factors have the same insoluble composition factor.
The computation shows that the two factors must contain $\mathrm{PSp}_6(2)$. 
Thus, $\overline{G_{uv}}=\overline{G_{vw}}=\overline{G_{v}}$.

\medskip
\textsc{Claim}. \(m=q-\epsilon \) is a power of \(2\).

Suppose, for a contradiction, that \(m\) is not a power of \(2\).
Then $m$ has an odd prime divisor $r>|\mathrm{Out}(L)|=df$; this is true if $f=1$, and for $f\geqslant 2$, we take $r\in  \ppd(p,2f)$.
Let \(R=\mathbf{\Omega}_r(M)\cong \C_r^7\).
By {\sc Magma} \cite{magma} computation (using the same computational method as in \cite[Subsection 2.2 and Lemma 3.31]{ex1}), we conclude that the Weyl subgroup $ \mathrm{PSp}_6(2)$ acts irreducibly on $R$.
Since $\overline{G_{uv}} =\mathrm{PSp}_6(2)$, we conclude that $G_{uv}\cap R=1$ or $R$.  

First, assume that $G_{uv}\cap R=1$. Then $|G_{uv}\cap M|_r=1$. 
It follows from  $f_r=1$ and $r$ is odd that $|G_{uv}|_r =|\mathrm{PSp}_6(2)|_r$.
On the other hand,  $|G_v|_r=|L_v|_r \geqslant  r^7|\mathrm{PSp}_6(2)|_r$. 
Note that $G_v=G_{uv}G_{vw}$, whence $|G_v|_r\leqslant |G_{uv}|_r^2$.
Consequently, $|\mathrm{PSp}_6(2)|_r\geqslant r^7$, which is impossible since $|\mathrm{PSp}_6(2)|=2^{9}\cdot 3^4\cdot 5\cdot 7$.

Assume now that $G_{uv}\cap R=R$ and so $G_{uv}\geqslant R$. 
Note that $R$ is the unique normal subgroup of $G_v$ isomorphic to $\mathrm{C}_r^7$.
Now $G_{vw}=G_{uv}^g\trianglerighteq R^g$.
Since $\overline{G_{vw}}=\mathrm{PSp}_6(2)$, the group $R^g$ is contained in $\Rad(G_{vw})\leqslant M.2.f$, and so $R^g =R$, contradicting Lemma \ref{factor}.
Thus, the Claim is proved.
 
By the above Claim, we let $m=q-\epsilon=2^a$ for some integer \(a\). 
Recall  $M=G_{v} \cap \C_m^{7}\geqslant \C_{2^a}^{7}/2$.
By the condition $q\geqslant 5$ if $\epsilon=+1$, we have $a\geqslant 2$. 
Thus, $\Omega_2(M)\cong \C_{2}^7$.
Set $R=\Omega_2(M)$. 
From the action of $W$ on $M$, we may view $R$ as a $\mathbb{F}_2\mathrm{PSp}_6(2)$-module.
Computation in  {\sc Magma}~\cite{magma} shows that $R$ is reducible with a submodule $U\cong \C_{2} $ while $R$ is not decomposable.
Moreover, the quotient module $R/U\cong \C_{2}^6$ is irreducible.
Therefore, $\langle x^{\mathrm{PSp}_6(2)} \rangle=R$ holds for any $ x\in R\setminus U$. 
Thus, the only submodules of \(\PSp_6(2)\) on \(R\) are \(0,U\) and \(R\).

Suppose that $|G_{uv} \cap M|>2^a$. Then $ G_{uv} \cap R \geqslant \mathrm{C}_2^2 $, and so $G_{uv}\cap R$ contain an element not in $U$.
Since $\overline{G_{uv}}=\mathrm{PSp}_6(2)$, by considering the action of $\mathrm{PSp}_6(2)$ on $R$, we conclude that $R\leqslant G_{uv}$. 
 Now $G_{vw}=G_{uv}^g\trianglerighteq R^g\cong \mathrm{C}_2^7$. 
If $R^g\leqslant M$, then $R^g=R$ as $R=\mathbf{O}_2(M)$, which leads a contradiction to  Lemma~\ref{factor}.
Therefore, $R^g$ is not contained in $M$.
Then 
\[ 
1\neq R^g/(R^g\cap M)\cong R^g(G_{vw}\cap M)/(G_{vw}\cap M) \unlhd G_{vw}/(G_{vw}\cap M)\cong G_{vw}M/M.
\] 
Note that $\PSp_6(2)\leqslant G_{vw}M/M\leqslant W.f$ and so $\mathbf{O}_{2}(G_{vw}M/M)\leqslant \mathrm{C}_2^2$. Since \(R^g/(R^g\cap M)\cong R^gM/M\leqslant\mathbf{O}_2(G_{vw}M/M)\leqslant \C_2^2\), it follows that  $R^g\cap M\geqslant \mathrm{C}_2^5$, which implies $R^g\cap R\geqslant \mathrm{C}_2^5$.
Since $R$ is normal in $G_v$, the group
 $R^g\cap R\trianglelefteq G_{vw}$ and therefore \(R\cap R^g\) is a submodule of \(\overline{G_{vw}}=\PSp_6(2)\) on \(R\). Considering the order of \(R^g\cap R\), we verify that $R^g\cap R=R$ and so \(R^g=R\), contradicting Lemma~\ref{factor}. 
Therefore, $|G_{uv} \cap M| \leqslant 2^a$. This implies that 
 \[2^{7a+10}=|L_v|_2\leqslant|G_v|_2\leqslant|G_{uv}|^2_2\leqslant |G_{uv} \cap M|^2_2\cdot |W|^2_2\cdot f_2=2^{20+2a}.\]

Consequently, $1\leqslant a\leqslant2$ and so \(q-\epsilon=2^a\) for \(a\in\{1,2\}\). Thus, 
\begin{equation}\label{eq:guvw}
 |G_{uvw}|=\frac{|G_{uv}||G_{vw}|}{|G_{v}|}=\frac{2^{20+2a}\cdot 3^8\cdot 5^2\cdot7^2}{2^{7a+10}\cdot 3^4\cdot 5\cdot7}=2^{10-5a}\cdot3^4\cdot 5\cdot7
\end{equation}
and so \(|\overline{G_{uvw}}|\) is divisible by \(3^4\cdot5\cdot7\). By checking \textsc{magma} \cite{magma} all the subgroups of \(\PSp_6(2)\) whose orders are divisible by \(3^4\cdot5\cdot7\), we conclude that \(\overline{G_{uvw}}=\PSp_6(2)\). Thus, \(|G_{uvw}|_2\geqslant |\overline{G_{uvw}}|_2=|\PSp_6(2)|_2=2^9\), contradicting Eq \eqref{eq:guvw} that \(|G_{uvw}|_2\leqslant 2^{10-5a}\leqslant 2^5\). Hence, the result is proved.


\end{proof}

\medskip\noindent
\emph{Proof of Proposition \ref{max rk 7}:} This immediately follows from Lemmas \ref{lm:e7cases1-7}--\ref{lm:e7-qpm1}.
\subsubsection{\texorpdfstring{$L=E_8(q)$}{L=E8(q)}}

In this subsection, we address the cases in which \(L_v\) is a subgroup of maximal rank of \(E_{8}(q)\) and prove the following result. 
\begin{proposition}\label{prop: e8 max rk}
    Suppose that Hypothesis \ref{hy:1} holds with \(L=E_8(q)\). Suppose that \(L_v\) is a maximal subgroup of maximal rank. Then \(s\leqslant2\).
\end{proposition}

We will use the following notation. Let \(d=(2,q-1)\), \(\epsilon=\pm1\), \(e=(3,q-\epsilon)\), \(t=(q-\epsilon)/e\) and \(h=(5,q-\epsilon)\). By \cite[Tables 5.1 and 5.2]{LSS}, \(L_v\) is one of the following:
\begin{table}[h]
    \centering
    \caption{Maximal subgroups of maximal rank of \(E_8(q)\)} 
    \begin{tabular}{@{}lll@{}}
\toprule
  Case&\(L_v\)       & Remarks \\
  \midrule
  (1)&  $d.\POmega_{16}^+(q).d$     &\\
  (2)&$t.\PSL^\epsilon_{9}(q).e.2$&\\
(3)& \(\SU_{5}(q^2).4\)&\\
(4)&\(\PGU_{5}(q^2).4\)&\\
(5)&\(d^2.\POmega_8^+(q^2).d^2.(\Sy_3\times 2)\)&\\
(6)& \(\PSU_3(q^4).8\)&\\
(7)& $d.(\PSL_2(q)\times E_7(q)).d$&\\
(8)&$e.(\PSL^\epsilon_{3}(q)\times E_6^\epsilon(q)).e.2$& \\
(9)&\(h.(\PSL^\epsilon_{5}(q))^2.h.4\)& \\
(10)&\(d^2.(\POmega_8^+(q))^2.d^2.(\Sy_3\times 2)\)&\\
(11)&\((\PSU_3(q^2))^2.8\)&\\
(12)&\(({}^3D_4(q))^2.6\)&\\
(13)&\(e^2.(\PSL_3^\epsilon(q))^4.e^2.\GL_2(3)\)& \\
(14)&\(d^4.(\PSL_2(q))^8.d^4.\AGL_3(2)\)& \(q>2\)\\
(15)& \((q^4+\epsilon q^3+q^2+\epsilon q+1)^2.(5\times\SL_2(5))\)&\\

(16)& \((q^4-q^2+1)^2.(\C_{12}\circ \mathrm{Q}_8.\mathrm{S}_3)\)&\\
(17)& \((q^8+\epsilon q^7-\epsilon q^5-q^4-\epsilon q^3+\epsilon q+1).\C_{30}\)&\\

(18)&\((q^2+1)^4.(4\circ2^{1+4}).\A_6.2\)&\\
(19)& \((q^2+\epsilon q+1)^4.(2.(3\times\PSU_4(2)))\)& \(q>2\) if \(\epsilon=-1\)\\

(20)& \((q-\epsilon)^8.W(E_8)\)& \(q\geqslant5\) if \(\epsilon=+1\)\\

\bottomrule
    \end{tabular}
    \label{tab:E8maximalrank}
\end{table}

\begin{remark}
In Case (16) of Table~\ref{tab:E8maximalrank}, the structure of $L_v$ in \cite[Table 5.2]{LSS} is given as $(q^4-q^2+1)^2.(\C_{12}\circ\GL_2(3))$.
Using the {\sc Magma} code in \cite[Subsection 2.2]{ex1} one can construct the factor group $L_v/(q^4-q^2+1)^2$. However, computation in {\sc Magma}  shows that the factor group $L_v/(q^4-q^2+1)^2$ contains no normal subgroup isomorphic to $\GL_2(3)$. Rather, it has the structure $\C_{12}\circ (\mathrm{Q}_8.\mathrm{S}_3)$, where the group $\mathrm{Q}_8.\mathrm{S}_3$ is isomorphic to a $\mathcal{C}_6$-subgroup of $\mathrm{SL}_2(7)$ as listed in \cite[Table 8.1]{holt}.
\end{remark}

\begin{lemma}\label{lm:e8case1-6}
    Suppose that Hypothesis \ref{hy:1} holds and \(L=E_8(q)\). Then cases (1)--(6) do not occur.
\end{lemma}
\begin{proof}
    Suppose, for a contradiction, that one of the cases (1)--(6) occurs. Then \(G_v\) has a normal subgroup \(M=G_{v}^{(\infty)}\) and  \(\ICF(G_v)=\{T\}\) such that \((M,T)\) is one of the following:
    \begin{enumerate}
        
        \item  \((M,T)=(\Omega_{16}^+(q),\POmega_{16}^+(q))\);
        \item  \((M,T)=t.\PSL_9^\epsilon(q),\PSL_9^\epsilon(q))\);
        \item  \((M,T)=(\SU_5(q^2),\PSU_5(q^2))\);
        \item  \((M,T)=(\PSU_5(q^2),\PSU_5(q^2))\);
        \item  \((M,T)=d^2\POmega_8^+(q^2).d^2.(\Sy_3\times2),\POmega_8^+(q^2))\);
        \item  \((M,T)=(\PSU_3(q^4),\PSU_3(q^4))\)
    \end{enumerate}
    
   Thus, \cite[Lemma 3.3]{small} applies, and we note that, for cases (1)--(4) and (6),
   \[T\notin\{\A_{6}, M_{11}, M_{12}, \PSp_{4}(2^{f}),\POmega_{8}^{+}(q),\PSU_{3}(8),\PSU_{4}(2)\}.\]
 This implies \(T\in\ICF(G_{uv})\cap\ICF(G_{vw})\). Hence, \(G_{uv}\) and \(G_{vw}\) contain \(M\). Note that there is no other subgroup of \(G_v\) isomorphic to \(M\) and therefore \(M^g=M\), contradicting the lemma \ref{factor}.

   For case (5), note that \(\POmega_8^+(q).\Sy_3\leqslant\overline{G_v}:=G_v/\Rad(G_v)\leqslant\Aut(\POmega_8^+(q))\). In particular, \(\overline{G_v}\) contains a triality automorphism of \(\POmega_8^+(q)\) and so \(\overline{G_v}\nleqslant\mathrm{P}\Gamma\mathrm{O}_8^+(q)\). Then by \cite[Lemma 3.3]{small}, \(\Gamma\) is not \((G,2)\)-arc-transitive. Thus, the result is proved.
\end{proof}
\begin{lemma}
    Suppose that Hypothesis \ref{hy:1} holds and \(L=E_8(q)\). Then cases (7)--(8) do not arise.
\end{lemma}
\begin{proof}
     Suppose for a contradiction that \(L_v\) is of the cases (7)--(8). Then \(L_v\) has a unique subgroup \(M\) isomorphic to \(d.E_7(q)\) or \(e.E_6^\epsilon(q)\) for \(\epsilon=\pm1\). Let \(C:=\C_{G_{v}}(M)\). Then  \(|\Out(L)|=(2,q-1)f\), and \(\SL_2(q)\leqslant C\leqslant \SL_2(q).\Out(L)\) for case (xiv) and \(e.\PSL_3^\epsilon(q)\leqslant C\leqslant e.\PSL_3^\epsilon(q).\Out(L)\) for case (xvii). Let \(\overline{H}:=HC/C\) for any subgroup \(H\leqslant G_{v}\). Then \(\overline{G_{uv}}\,\overline{G_{vw}}=\overline{G_{v}}\) is almost simple with socle \(T=E_7(q)\). By \cite[Table 5]{LPS}, at least one of \(\overline{G_{uv}}\) or \(\overline{G_{vw}}\) contains \(T\). Without loss of generality, assume that \(\overline{G_{uv}}\geqslant T\). Then, \(G_{uv}\triangleright M\). Since \(G_{uv}^{g}=G_{vw}\), we deduce that \(M^{g}\) is also a normal subgroup of \(G_{vw}\). However, \(M\) is the unique subgroup of \(G_{v}\) isomorphic to \(d.E_7(q)\) or \(e.E_6(q)\). Hence, \(M^{g}=M\), contradicting Lemma~\ref{factor}. Thus, the result is proved.   
\end{proof}

\begin{lemma}
    Suppose that Hypothesis \ref{hy:1} holds and \(L=E_8(q)\). Suppose that \(L_v\) is one of the cases (9)--(12). Then \(s\leqslant2\).
\end{lemma}
\begin{proof} 
    Suppose for a contradiction that \(s\geqslant3\). Then \(\Gamma\) is \((L,2)\)-arc-transitive and so \(L_v=L_{uv}L_{vw}\). Note that \(L_v\) has a unique subgroup \(M\) and \(M/Z(M)\cong M_1\times M_2\) such that \(M_1\cong M_2\) and \((M,M_1)\) is one of the following:
    \begin{itemize} 
        \item Case (9), \((M,M_1)=(h.(\PSL_5^\epsilon(q))^2,\PSL_5^\epsilon(q))\);
        \item Case (10), \((M,M_1)=(d^2.(\POmega_8^+(q))^2,\POmega_8^+(q))\);
        \item Case (11),  \((M,M_1)=((\PSU_3(q^2))^2,\PSU_3(q^2))\);
        \item Case (12), \((M,M_1)=({}^3D_4(q)^2,{}^3D_4(q))\).
    \end{itemize}

    Let \(Z\) be the centre of \(M\) and \(\overline{S}:=SZ/Z\) for any subgroup \(S\leqslant G_v\). Let \(\varphi_i\) be the projection from \(\overline{M}\) to \(M_i\) for \(i=1,2\). Then \(\varphi_i(\overline{L_{v}})=\Aut(\POmega_8^+(q))\) for \(i=1,2\). Note that \(\ppd(p,if)\neq \varnothing\) for \(i\in\{3,4,5,8,10,12\}\). Let \(r_i\in\ppd(p,if)\) for \(i\in\{3,4,5,8,10,12\}\). Let 
    \[
    J:=\begin{cases}
    \{r_3,r_4,r_5\}&\text{case (9), \(\epsilon=+1\);}\\
    \{r_8,r_{10}\}&\text{case (9), \(\epsilon=-1\);}\\
        \{r_3,r_{4},r_6\}&\text{case (10);}\\
        \{p,r_{12}\}&\text{case (11), \(p\geqslant5\);}\\
        \{r_4,r_{12}\}&\text{case (11), \(p=2,3\);}\\
       \{r_3,r_{12}\}&\text{case (12);}\\
    \end{cases}
    \]

    Then \(J\subseteq\pi(M_1)\). 
    Note that \(d\mid2\) and \(h\) divides \(q^2-1=p^{2f}-1\) and so \(d_{r_i}=h_{r_i}=1\) for \(i\in\{3,4,5,8,10\}\). Thus, \(|\overline{M}|_{r_i}=|M|_{r_i}\) for \(i\in\{3,4,5,8,10\}\).

\medskip

    \textsc{Claim:} \(J\subseteq \pi(\overline{L_{uv}}\cap M_i),\pi(\overline{L_{vw}}\cap M_i)\) for \(i=1,2\).

    Assume, for a contradiction, that the assertion is false. Then there exists some \(r\in J\) and \(i\in\{1,2\}\) such that \(|\overline{L_{uv}}\cap M_i|_{r}=1\). Then
    \[
    |L_{uv}|_{r}=|L_{uv}\cap M|_{r}4_{r}=|\overline{L_{uv}}\cap\overline{M}|_{r}\leqslant|\varphi_{3-i}(\overline{L_{uv}}\cap\overline{M})|_{r}=|M_{3-i}|_{r}.
    \]
    On the other hand, \(|L_{v}|_{r_j}=|M_1|^2_{r_j}\). It follows from Lemma \ref{3ATprime} that
    \[
    |M_1|^4_{r}=|L_v|_{r}^2\leqslant |L_{uv}|_{r}^3|\Out(L)|_{r}\leqslant |M_{3-i}|^3_{r}|\Out(L)|_{r}=|M_1|_{r}^3|\Out(L)|_{r}.
    \]
    Thus, \(|\Out(L)|_{r}\geqslant |M_1|_{r}\). However, \(|\Out(L)|_{r}=((2,q-1)f)_{r}=1\). Thus, \(J\subseteq \Pi(\overline{L_{uv}}\cap M_i)\) and with an identical argument, \(J\subseteq \Pi(\overline{L_{vw}}\cap M_i)\). Hence, the claim is proved.

    Therefore, \(|\overline{L_{uv}}\cap M_i|_{r}>1\) for \(r\in J\) and \(i=1,2\). Then, for cases (9), (11)--(12), it follows from \cite[Tables 10.3 and 10.5]{transitive} that \(M_i\leqslant \overline{L_{uv}}\) for \(i=1,2\) and so \(\overline{M}\leqslant\overline{L_{uv}}\). Thus, \(M^g\leqslant L_{vw}\). However, \(M\) is the unique subgroup of \(L_v\) isomorphic to \(d.(M_1\times M_2)\) and so \(M^g=M\), contradicting Lemma \ref{factor}. 
    
    Now, we deal with case (8). Then \(L_v=H.2\) where \(H=\mathrm{Spin}_8^{+}(q) \circ \mathrm{Spin}_8^{+}(q)=d^2.(\POmega_8^+(q))^2.d^2.\Sy_3\). Note that $L_v$ arises from a subsystem $ \Delta=2D_4 $ of $E_8$ (one $D_4$ has simple roots $\{\alpha_2,\alpha_3,\alpha_4,\alpha_5\}$), and in $W_{\Delta}=2\times \mathrm{S}_3$,  the group $2$  swaps two $D_4$s  and the group $\mathrm{S}_3$ induces the full group of graph automorphisms of each $D_4$.
Thus, $L_v$ acts transitively on the two components $\mathrm{Spin}^{+}_8(q)$ and so \(\overline{L_{v}}\) acts transitively on \(M_1\) and \(M_2\). Note that \((L_{uv}H/H)(L_{vw}H/H)=L_{v}/H=\C_2\). Thus, at least one of \(L_{uv}H/H\) or \(L_{vw}H/H\) is \(\C_2\). Without loss of generality, assume that \(L_{uv}H/H=\C_2\) and therefore \(\overline{L_{uv}}\) acts transitively on \(\overline{L_{uv}}\cap M_1\) and \(\overline{L_{uv}}\cap M_2\).

   On the other hand, since \(J\subseteq \Pi(\overline{L_{uv}}\cap M_i)\) for \(i=1,2\), it follows from \cite[Table 10.1]{transitive} that \(\Omega_7(q)\trianglelefteq\overline{L_{uv}}\cap M_i\leqslant \Omega_7(q).2\) or \(M_i\leqslant \overline{L_{uv}}\) for \(i=1,2\). If \(\Omega_7(q)\trianglelefteq\overline{L_{uv}}\cap M_i\) for some \(i\in\{1,2\}\), then since \(\overline{L_{uv}}\) acts transitively on \(\overline{L_{uv}}\cap M_1\) and \(\overline{L_{uv}}\cap M_2\), \(\Omega_7(q)\trianglelefteq\overline{L_{uv}}\cap M_i\) for \(i=1,2\). Note that \(\overline{L_{uv}}\cap M_i\trianglelefteq\varphi_i(\overline{L_{uv}})\) for \(i=1,2\). Thus, \(\Omega_7(q)\trianglelefteq\varphi_i(\overline{L_{uv}})\leqslant \Omega_7(q).2\) for \(i=1,2\). Let \((p^{3f}-1)_{r_3}=r_3^a\). Then \(
   |\POmega^+_8(q)|_{r_3}=(p^{3f}-1)^2_{r_3}=r_3^{2a}\) and \(|\Omega_7(q)|_{r_3}=(p^{3f}-1)_{r_3}=r_3^a\). Thus, \(|L_{v}|_{r_3}=r_3^{4a}\) and \(|L_{uv}|_{r_3}\leqslant|\Omega_7(q)|_{r_3}^2d^4_{r_{3}}|\Sy_3\times 2|_{r_3}=r_3^{2a}\).
   However, since \(|\Out(L)|_{r_3}=((2,q-1)f)_{r_3}=1\), it follows from Lemma \ref{3ATprime} that 
   \[
   r_3^{8a}=|L_{v}|_{r_3}^{2}\leqslant |L_{uv}|_{r_3}^3|\Out(L)|_{r_3}\leqslant r_3^{6a},
   \] 
which is a contradiction as \(a\geqslant 1\). Thus, the lemma is proved.
\end{proof}

\begin{lemma}
Suppose that Hypothesis~$\ref{hy:1}$ holds and \(L=E_8(q)\). Suppose that case (13) occurs. Then \(s\leqslant2\)
\end{lemma}
 
\begin{proof}  
Suppose for a contradiction that case (13) occurs with \(s\geqslant3\). Then \(\Gamma\) is \((L,2)\)-arc-transitive and therefore \(L_v=L_{uv}L_{vw}\).

Now $L_v$ arises from a subsystem $ 4A_2 $ of $E_8$ (two $A_2$ factors have simple roots $\{\alpha_1,\alpha_3\}$ and  $\{\alpha_5,\alpha_6 \}$).
Then $\mathrm{SL}_3(q)^2<L_v$ or $\mathrm{SU}_3(q)^2<L_v$. 
Computation on the Weyl group shows that the centre of $ \mathrm{GL}_2(3)\cong 2.\mathrm{S}_4$ swaps two simple roots in each $A_2$, and $\mathrm{GL}_2(3) $ acts on the four $A_2$s as $\mathrm{S}_4$ on $4$ points. 
Note that, for each factorisation of $ \mathrm{S}_4 $, at least one factor is transitive, and if one factor is transitive but not primitive then the other factor has order divisible by $3$.  
Let $M$ be the normal subgroup $(3,q\pm1)^2.\mathrm{PSL}^\epsilon_3(q)^4$ of $L_v$ and let $Z=(3,q\pm1)^4 $ be the centre of $M$. We denote by \(\overline{S}:=SZ/Z\) for any subgroup \(S\leqslant G_v\).

\smallskip
\textsc{Claim~1:} \(q\neq2\) if \(L_v=3^2.(\mathrm{PSU}_3(2)^4).3^2.\mathrm{GL}_2(3)\).

Assume now that $L_v=3^2.(\mathrm{PSU}_3(2)^4).3^2.\mathrm{GL}_2(3)$. Then $G_v=L_v$.
Since $Z(\mathrm{GL}_2(3))=\C_2$ swaps two simple roots of each $A_2$, it induces a field automorphism of $\mathrm{PSU}_3(2)$.
So $G_v/Z\leqslant \mathrm{P\Gamma U}_3(2) \wr \mathrm{S}_4$.
Notice that $|\overline{G_v}|=2^{16}\cdot 3^{11} $.
In {\sc Magma}, by computing subgroups $K$ of $ (\mathrm{P\Gamma U}_3(2) \wr \mathrm{S}_4)/(\mathrm{PSU}_3(2)^4)$  such that $\vert K\vert= 3^2\cdot\vert  \mathrm{GL}_2(3)\vert=2^4 \cdot 3^3  $ and $K/(K \cap \mathrm{P\Gamma U}_3(2)^4)\cong \mathrm{S}_4$, we can identify the group $\overline{G_v}$.
Note that $\mathrm{PSU}_3(2)\cong 3^2:\mathrm{Q}_8$.
Thus, $\overline{G_v}$ has a normal $3$-subgroup $3^8$, and computation shows that it is exactly $\mathbf{O}_3(\overline{G_v})$.
Then $M_1:=\mathbf{O}_3(\overline{G_v})=Z.3^8$.
In the factorisation $(\mathrm{Q}_8)^4.3^2.\mathrm{GL}_2(3)=G_v/M_1=(G_{uv}M_1/M_1)(G_{vw}M_1/M_1)$, the Sylow $2$-subgroups of two factors are isomorphic and have order at least $2^8$.
By the {\sf LowIndexSubgroups} command of \textsc{Magma} \cite{magma} we can obtain all candidates for the factorisation $G_v/M_1=(G_{uv}M_1/M_1)(G_{vw}M_1/M_1)$,  and it turns out that both factors contains the Sylow $2$-subgroup of $G_v/M_1$ (which is of order $2^{16}$).
Now, in the factorisation $ \overline{G_v}=\overline{G_{uv}}\,\overline{G_{vw}}$, both of two factors \(\overline{G_{uv}}\) and $ \overline{G_{vw}}$  have order divisible by $2^{16}\cdot 3^4$.
Again, by the {\sf LowIndexSubgroups} command, we can obtain all candidates for the factorisation $ \overline{G_v}=\overline{G_{uv}}\,\overline{G_{vw}}$.
There are four such factorisations, and they satisfy the following:
\begin{enumerate}
\item[(1)]  $\mathbf{O}_3(\overline{G_{uv}})=\mathbf{O}_3(\overline{G_{vw}})=\mathbf{O}_3(\overline{G_v})=3^8$.
\item[(2)] In three factorisations, $\overline{G_{uv}}=\overline{G_v}$, and the Sylow $3$-subgroup of \(\overline{G_{vw}}\) has order $3^8$, $3^{9}$ and $3^{10}$, respectively. In the other factorization, $|\overline{G_{uv}})|=2^{16}\cdot 3^{10}$, and $|\overline{G_{vw}})|=2^{16}\cdot 3^{9}$.
\item[(3)] In the homomorphism the $G_v/Z \to \mathrm{GL}_2(3) $,  the image of \(\overline{G_{uv}}\) is $\mathrm{GL}_2(3)$ and the image of \(\overline{G_{vw}}\) is $\mathrm{SD}_{16}$. 
\end{enumerate}

For (2) we see that $G_{vw} \cap Z>1$.
Computation on the Weyl group shows that both $ \mathrm{SD}_{16}$ and $\mathrm{GL}_2(3)$  acts irreducibly on $Z=3^2$ (Note that $Z$ is in a $\sigma$-stable torus. Moreover, computation shows that $\N_W(W(\Delta))$ has a group $\mathrm{GL}_2(3)$ such that $\N_W(W(\Delta))=W(\Delta):\mathrm{GL}_2(3)$, and so we can compute the action of the group $\mathrm{GL}_2(3)$ on $Z$).
Thus, it follows from (3) that $Z \leqslant G_{vw} $. Then (1) implies $ \mathbf{O}_3(G_{vw})=\mathbf{O}_3(G_v)$.
Since $G_{uv}^g=G_{vw}$, the group $ \mathbf{O}_3(G_v)^{g^{-1}} \leqslant \mathbf{O}_3(G_{uv})$ and then (1) implies $\mathbf{O}_3(G_{uv})=\mathbf{O}_3(G_v)=\mathbf{O}_3(G_{vw})$.
It contradicts Lemma~\ref{factor}. 

\medskip

Now, \(q\geqslant3\). For \(1\leqslant i\leqslant 4\), let $N_i\cong \mathrm{PSL}^\epsilon_3(q)$ be the $4$ components of $N:=\mathrm{PSL}^\epsilon_3(q)^4\leqslant \overline{G_{v}}$, and let $\varphi_i$ be the projection from $N$ to \(N_i\). 
Then the induced action of $\overline{L_v}\cong\GL_2(3)$ on $\{N_1,\ldots,N_4\}$ is \(\Sy_4\). Since $L_{v}/M=(L_{uv}M/M)(L_{vw}M/M)$ and at least one factor is transitive, we may assume, without loss of generality, that the induced action of \(L_v/M\) is transitive and therefore \(\varphi_{1}(\overline{L_{uv}}\cap N)\cong\cdots\cong\varphi_{4}(\overline{L_{uv}}\cap N)\). Thus, \(|\overline{L_{uv}}\cap N|\) divides \(|\varphi_1(\overline{L_{uv}}\cap N)|^4\). 

 \medskip
 \textsc{Claim~2:} $\varphi_1(\overline{L_{uv}}\cap N)= N_1$. 
 
Since \(|Z|=(3,q\pm1)\) and \(|\GL_2(3)|=2^4\cdot3\), it follows that \(|L_v|_r=|N|_r\) for any prime \(r\in\Pi(L_v)\setminus\{2,3\}\). Thus
\[
\pi(\overline{L_{uv}}\cap N)\cup\{2,3\}=\pi(\overline{L_{uv}})\cup\{3\}=\pi(L_{uv})=\pi(L_{v}).
\]
On the other hand, note that \(|N_1|_2=|\PSL_3^\epsilon(q)|_2\geqslant 2^8\) and \(|L_v|_2=|N_1|^4_2|\GL_2(3)|_2=2^4|N_1|_2^4\). Thus,
\[
|\overline{L_{uv}}\cap N|_2\geqslant\frac{|\overline{L_{uv}}|_2}{|\GL_2(3)|_2}=
\frac{|\overline{L_{uv}}|_2}{2^4}=\frac{|L_{uv}|_2}{2^4}\geqslant \frac{|L_{v}|_2^{\frac{1}{2}}}{2^4}=\frac{(2^4|N_1|_2^4)^{\frac{1}{2}}}{2^4}=2^{-2}|N_1|_2^2>1\]
and so \(2\in\Pi(\overline{L_{uv}}\cap N)\). Similarly, note that \(|Z|_3^3\leqslant |N_1|_3\), \(|N_1|_3=|\PSL_3^\epsilon(q)|_3\geqslant3\) and
\[
|L_v|_2=|Z|_3^4|N_1|^4_3|\GL_2(3)|_3=3|Z|_3^4|N_1|_3^4.
\]
Thus,
\[
|\overline{L_{uv}}\cap N|_3\geqslant\frac{|\overline{L_{uv}}|_3}{|Z|_3^2|\GL_2(3)|_3}=
\frac{|\overline{L_{uv}}|_3}{3|Z|_3^2}\geqslant\frac{|L_{uv}|_3}{3|Z|_3^4}\geqslant \frac{|L_{v}|_3^{\frac{1}{2}}}{3|Z|_3^4}=\frac{(3|Z|_3^4|N_1|_3^4)^{\frac{1}{2}}}{3|Z|_3^4}=\frac{|N_1|_3^2}{3^{\frac{1}{2}}|Z|_3^4}\geqslant\frac{|N_1|_3^2}{3^{\frac{1}{2}}|N_1|_3^{\frac{4}{3}}}>1\]     
 as \(|N_1|_3>1\) and so \(3\in\Pi(\overline{L_{uv}}\cap N)\). Now, we have
 \[
\Pi(\varphi_1(\overline{L_{uv}}\cap N))= \Pi(\overline{L_{uv}}\cap N)=\Pi(L_v)\supseteq\Pi(N_1).
 \]
 This together with \cite[Table 10.7]{transitive} implies that either (1) \((\epsilon,q,\varphi_1(\overline{L_{uv}}\cap N))=(-,3,\PSL_2(7))\), or (2) \((\epsilon,q,\varphi_1(\overline{L_{uv}}\cap N))=(-,5,\A_7)\), or (3) \(\varphi_1(\overline{L_{uv}}\cap N)=N_1\).
 
 Assume that case (1) occurs. Then 
 \[|L_{uv}|^2_3\leqslant (|\PSL_2(7)|_3^4|\GL_2(3)|_3)^2=3^{10}<|\PSU_3(3)|_3^4|\GL_2(3)|_3=3^{13}=|L_v|_3,\]
 contradicting the fact that \(L_v=L_{uv}L_{vw}\) is a homogeneous factorisation. 

 Then assume that case (2) occurs. Then
 \[|L_{uv}|^2_5\leqslant (|\A_7|_5^4)^2=5^8<|\PSU_3(5)|_5^4=5^{12}=|L_v|_5,\]
 contradicting the fact that \(L_v=L_{uv}L_{vw}\) is a homogeneous factorisation. Thus, \(\varphi_1(\overline{L_{uv}}\cap N)=N_1\) and Claim~2 is proved.

 Note that \(N_1\) is non-abelian simple and  \(\ICF(\overline{L_{uv}}\cap N)=\{\PSL_3^\epsilon(q)\}\). Thus, \(\overline{L_{uv}}\cap N\cong\PSL_3^\epsilon(q)^m\) for some \(m\leqslant4\).

 \medskip
 \textsc{Claim~3:} \(m=4\).

By Scott's Lemma \cite[Theorem 4.16]{csaba}, \(m\) divides \(4\). Suppose for a contradiction that \(m\leqslant2\). Note that \(\ppd(p,\frac{(3-\epsilon)3f}{2})\neq\varnothing\) (recall that \(\ppd(2,6)=\{7\}\)), and so there exists \(r\in \ppd(p,\frac{(3-\epsilon)3f}{2})\).
Then \(r>3fi\) and so \(|\Out(L)|_r=f_r=1\). Note that \(r>3\), it follows that \(|L_v|_r=|N_1|^4_r\) and \(|L_{uv}|_r=|\overline{L_{uv}}|_r\leqslant|N_1|_r^2\).
By Lemma \ref{3ATprime}, we have
\[
|N_1|_r^{8}=|L_v|_r^2\leqslant|L_{uv}|_r^3|\Out(L)|_r\leqslant |N_1|_r^{6},
\]
which is impossible. Thus, Claim 3 is proved.

Now, \(\PSL_3^\epsilon(q)^4\cong N\leqslant \overline{L_{uv}}\) and so \(Z.N=M\leqslant L_{uv}\). Thus, \(M^g\leqslant L_{vw}\). However, \(M\) is the unique subgroup of \(L_v\) isomorphic to \(Z.(\PSL_3^\epsilon(q))^4\). Therefore, \(M^g=M\), contradicting Lemma \ref{factor}. Hence, the result is proved.

\end{proof}

\begin{lemma}\label{lm:E8A18}
Suppose that Hypothesis~$\ref{hy:1}$ holds and  $L=E_8(q)$. Then case (14) does not occur.
\end{lemma}
\begin{proof}
Suppose, for a contradiction, that case (14) occurs. Then $L_v=d^4.\mathrm{PSL}_2(q)^8.d^4.\mathrm{AGL}_3(2)$ and \(G_v=L_v.m\) where \(m\mid f\), where $d=(2,q-1)$ and $q\geqslant 3$.
Now $L_v$ is a subgroup of maximal rank, and it arises from a subsystem $\Delta$ of type $ (A_1)^8$ of $E_8$. 
The subsystem $\Delta$ can be taken to be the one with simple roots $\{ \alpha_2,\alpha_3,\alpha_5,\alpha_7, \alpha_0, \beta_1,  \beta_2,  \beta_3\}$, where 
\begin{align*}
 &\alpha_0= 2\alpha_1 +3\alpha_2 +4\alpha_3 +6\alpha_4 +5\alpha_5 +4\alpha_6 + 3\alpha_7+2\alpha_2,\\
 &\beta_1=2\alpha_1+ 2\alpha_2 +3\alpha_3 +4\alpha_4 +34\alpha_5 +2\alpha_6 +\alpha_7, \\
 &\beta_2=\alpha_2 + \alpha_3 + 2\alpha_4 + 2\alpha_5+2\alpha_6+ \alpha_7,\\
 &\beta_3=\alpha_2 + \alpha_3 + 2\alpha_4 +  \alpha_5.
\end{align*}

Since a subsystem $(A_1)^4$ of $\Delta$ is generated by the four simple roots  $\alpha_2,\alpha_3,\alpha_5,\alpha_7$ of $E_8$, it follows from \cite[Proposition 2.6.2]{CFSG} that there exists $H=\mathrm{SL}_2(q)^4<E_8(q) $.
Let $ Z\cong \mathrm{C}_d^4$ be the centre of the group $H$. 
Let $M $ be the normal subgroup $d^4.\mathrm{PSL}_2(q)^8 $ of $G_v$, and denote by \(\overline{A}:=AZ/Z\) for any subgroup \(A\leqslant G_v\). 
Then $Z\leqslant M$ and $\overline{M}\cong \mathrm{PSL}_2(q)^8 $.

Now, assume that \(q\) is odd and so that \(Z\) is non-trivial. Then the centre of each $\mathrm{SL}_2(q)$ factor in $\mathrm{SL}_2(q)^4$ is not contained in $C$.

\smallskip
\textsc{Claim:}  Let $A \leqslant G_v$ and suppose that $q$ is odd. If $\overline{A}\geqslant\overline{M}$, then $A \geqslant M$.\quad

Since $\overline{A}\geqslant \overline{M}$, it suffices to show that $A \geqslant Z$. 
Since \(\PSL_2(q)^8\leqslant \overline{A}\), it follows that \(\PSL_2(q)^4=\overline{A}\cap\overline{H}\). By checking \textsc{Magma}, we find that there is no proper subgroup of \(\SL_2(3)^4\) containing \(\PSL_2(3)^4\) as a normal quotient. Thus, \(\SL_2(3)^4\cong H\leqslant A\), which is a contradiction. Now, \(q\geqslant5\). Then \(\SL_2(q)\) is quasisimple and \(\PSL_2(q)^4=\overline{A}\cap\overline{H}\), we deduce that \(Z\leqslant H\leqslant A\). Thus, \(Z<A\) and Claim~1 is proved.

\smallskip

Next, we divide our analysis into two cases: $q=3$ and $q>3$.  
 
\smallskip
\noindent{\bf Case 1.} Suppose that $q=3$.

Now $H=L$  and $\overline{G_v} $ is a subgroup of $\mathrm{S}_4\wr \mathrm{AGL}_3(2)$ with the index $2^4$.
Computation in {\sc Magma}\cite{magma} shows that $\mathrm{S}_4 \wr \mathrm{AGL}_3(2)$ has two classes of subgroups of indices $2^{4}$ which have epimorphisms to $\mathrm{AGL}_3(2)$, and they are the candidates for $\overline{G_v}$. 
Note that $\vert G_v\vert=2^{30}\cdot 3^9\cdot 7$.
Thus, $\vert G_{uv}\vert$ is divisible by $2^{15}\cdot 3^5\cdot 7$.  
In the factorisation $\overline{G_v}=  \overline{G_{uv}}\,\overline{G_{vw}}$, we may assume that $\vert \overline{G_{uv}} \vert_2\geqslant \vert \overline{G_{vw}} \vert_2$.
Then two factors satisfy the following conditions:  
\begin{itemize}
\item Both $\vert \overline{G_{uv}} \vert$ and $\vert \overline{G_{vw}} \vert$ are divisible by $2^{11}\cdot 3^5\cdot 7$.
\item $\overline{G_{uv}}/\mathbf{O}_2(\overline{G_{uv}})\cong  \overline{G_{vw}}/\mathbf{O}_2(\overline{G_{vw}})$, and  the Sylow $3$-subgroups of $ \overline{G_{uv}}$ and $\overline{G_{vw}} $ are isomorphic, and the Sylow $2$-subgroups of $\overline{G_{uv}}/\mathbf{O}_2(\overline{G_{uv}}) $ and $\overline{G_{vw}}/\mathbf{O}_2(\overline{G_{vw}}) $ are isomorphic.  
\item $\vert \overline{G_{uv}} \vert_2 \leqslant 2^4\cdot \vert \overline{G_{vw}} \vert_2$. 
\end{itemize}    
Computation in {\sc Magma}\cite{magma} shows that, for all possible factorisations $\overline{G_v}= \overline{G_{uv}}\,\overline{G_{vw}}$, it holds that both  $\overline{G_{uv}}$ and $\overline{G_{vw}}$ contain  $\overline{M}$ and  have epimorphisms to $\mathrm{AGL}_3(2)$. 
By Claim 1, we conclude that  both $G_{uv}$ and $G_{vw}$ contain $Z$. 
It follows from $G_{uv}^g=G_{vw}$ that $Z$ is normalised by $g$, contradicting Lemma~\ref{factor}.

\smallskip
\noindent{\bf Case 2.} Suppose that $q>3$. 
 
 Note that $\overline{G_{v}}\leqslant \mathrm{P\Gamma L}_2(q)\wr \mathrm{AGL}_3(2)$. Now, $M=G_v^{(\infty)}$ and $ \overline{M}\cong \mathrm{PSL}_2(q)^8\leqslant \overline{G_{v}}$.
 Let $\overline{M}_i\cong \mathrm{PSL}_2(q)$, \(1\leqslant i\leqslant 8\), be the $8$ components of $\overline{M}$, and let $R$ be the kernel of $\overline{G_v}$ acting on $\{\overline{M}_1,\ldots,\overline{M}_8\}$.
Then 
  $ \overline{G_{v}}/R \cong \mathrm{AGL}_3(2)$.
We consider the factorisation $ \overline{G_{v}}/R=  (\overline{G_{uv}}\,R/R) (\overline{G_{vw}}\,R/R)$.
Note that $|\mathrm{AGL}_3(2)|=2^6\cdot 3\cdot 7$.
By computation in \textsc{Magma}\cite{magma}, for each factorisation $ \AGL_3(2)=  (\overline{G_{uv}}\,R/R) (\overline{G_{vw}}\,R/R)$ of $\mathrm{AGL}_3(2)$, exchanging \(\overline{G_{uv}}\,R/R\) and \(\overline{G_{vw}}\,R/R\) if necessary, either \(\overline{G_{uv}}\,R/R\) is primitive, or \(\overline{G_{uv}}\,R/R\) is a transitive soluble subgroup of \(\AGL_3(2)\) and \(\overline{G_{vw}}\,R/R\) is intransitive with \(|\overline{G_{vw}}\,R/R|_7=7\). 

Note that \(\overline{G_{uv}}\,R/R\) is also transitive on \(\{\overline{M}_1,\ldots,\overline{M}_8\}\).  Let $\varphi_i$ be the projection of $\overline{M}$ to $\overline{M}_i$ for each $i\in \{1,\ldots,8\}$. 
Then $\varphi_1(\overline{G_{uv}} \cap \overline{M}) \cong \cdots \cong \varphi_8(\overline{G_{uv}} \cap \overline{M})$, and  so 
 \(\vert \overline{G_{uv}} \cap \overline{M} \vert\)   divides  \(\vert \varphi_1(\overline{G_{uv}}\cap \overline{M})\vert^8\). 
 
By similar argument  of Lemma~\ref{E7weyl}, we conclude that 
\begin{equation}\label{eq:8A1-1}
|\overline{G_{uv}} \cap \overline{M}|_r\geqslant \left(\frac{ |\mathrm{PSL}_2(q)| ^{4} }{d^4\cdot f \cdot (2^6\cdot 3\cdot 7)^{1/2}}\right)_r \text{ for any prime $r$}, 
\end{equation} 

and  $\varphi_i(\overline{G_{uv}} \cap \overline{M}) =\mathrm{PSL}_2(q)$ for \(1\leqslant i\leqslant8\).

Let \(\ell=\m_{\overline{G_{uv}} \cap \overline{M}}(\PSL_2(q))\). Then since \(\overline{G_{uv}}\,R/R\) is soluble, it follows that \(\m_{G_{uv}}(\PSL_2(q))=\m_{\overline{G_{uv}}\cap\overline{M}}(\PSL_2(q))+\m_{\overline{G_{uv}}\,R/R}(\PSL_2(q))=\ell\). By Scott's Lemma \cite[Theorem 4.16]{csaba}, \(\ell\) divides \(8\) such that \(\overline{G_{uv}}\,R/R\) preserves a partition of \(\{\overline{M}_1,\ldots,\overline{M}_8\}\) with \(\ell\) parts of size \(8/\ell\). Note that Eq. \eqref{eq:8A1-1} implies that \(
|\overline{G_{uv}} \cap \overline{M}|_p\geqslant q^3\).

Thus, \(q^\ell=|\PSL_2(q)|^\ell=|\overline{G_{uv}} \cap \overline{M}|_p\geqslant q^3\) and therefore \(\ell\geqslant3\). Since \(\ell\) also divides \(8\), we conclude that \(\ell=4\) or \(8\).

If \(\ell=8\), then \(\overline{M}\leqslant \overline{G_v}\) and it follows from Claim that \(M\leqslant G_{uv}\). Thus, \(M^g\leqslant G_{vw}\). However, \(M\) is the unique subgroup of \(G_v\) isomorphic to \(d^4.\PSL_2(q)^8\) and this would imply \(M^g=M\), contradicting Lemma \ref{factor}.

Hence, \(\ell=4\). This implies that \(\overline{G_{uv}}\,R/R\) preserves a partition of \(\{\overline{M}_1,\ldots,\overline{M}_8\}\) with \(4\) parts. Therefore, \(\overline{G_{uv}}\,R/R\) is not primitive and \(|\overline{G_{vw}}\,R/R|_7=7\).
Reordering \(\{\overline{M}_1,\ldots,\overline{M}_8\}\) if necessary, we assume that $\overline{G_{vw}}$ is transitive (and primitive) on $\{\overline{M}_1,\ldots,\overline{M}_7\}$ and so \(\varphi_1(\overline{G_{vw}} \cap \overline{M}) \cong \cdots \cong \varphi_7(\overline{G_{vw}} \cap \overline{M})\). Now, since \(4=\m_{G_{uv}}(\PSL_2(q))=\m_{G_{vw}}(\PSL_2(q))\) we have
\[
\begin{split}
4=\m_{G_{vw}}(\PSL_2(q))&=\m_{\overline{G_{vw}}\,R/R}(\PSL_2(q))+\m_{G_{vw}\cap M}(\PSL_2(q))\\
&=\m_{\overline{G_{vw}}\,R/R}(\PSL_2(q))+\m_{\overline{G_{vw}}\cap \overline{M}}(\PSL_2(q))\\
&\leqslant\m_{\overline{G_{vw}}\,R/R}(\PSL_2(q))+\sum_{i=1}^8\m_{\varphi_i(\overline{G_{vw}}\cap \overline{M})}(\PSL_2(q))\\
&=\m_{\overline{G_{vw}}\,R/R}(\PSL_2(q))+7\m_{\varphi_1(\overline{G_{vw}}\cap \overline{M})}(\PSL_2(q))+\m_{\varphi_8(\overline{G_{vw}}\cap \overline{M})}(\PSL_2(q))\\
\end{split}
\]

Since \(\overline{G_{vw}}\,R/R\leqslant \AGL_3(2)\) and \(\overline{G_{vw}}\cap \overline{M}_8\cong\PSL_2(q)\), it follows that \(\m_{\overline{G_{vw}}\,R/R}(\PSL_2(q))+\m_{\overline{G_{vw}}\cap \overline{M}_8}(\PSL_2(q))\leqslant2\). Thus, \(7\m_{\varphi_1(\overline{G_{vw}}\cap \overline{M})}(\PSL_2(q))\geqslant 2\) and so \(\m_{\varphi_1(\overline{G_{vw}}\cap \overline{M})}(\PSL_2(q))\geqslant1\). This implies that \(\varphi_i(\overline{G_{vw}}\cap \overline{M})=\PSL_2(q)\) for \(1\leqslant i\leqslant 7\). Since the induced action of \(\overline{G_{vw}}\,R/R\) on \(\{\overline{M}_1\ldots,\overline{M}_7\}\) is primitive, we deduce from Scott's Lemma \cite[Theorem 4.16]{csaba} that \(\m_{\overline{G_{vw}}\cap\prod_{i=1}^7\overline{M}_i}(\PSL_2(q))=1\) or \(7\). If \(\m_{\overline{G_{vw}}\cap\prod_{i=1}^7\overline{M}_i}(\PSL_2(q))=1\), then 

\[
4=\m_{G_{vw}}(\PSL_2(q))\leqslant \m_{\overline{G_{vw}}\,R/R}(\PSL_2(q))+\m_{\varphi_8(\overline{G_{vw}}\cap \overline{M})}(\PSL_2(q))+\m_{\overline{G_{vw}}\cap\prod_{i=1}^7\overline{M}_i}(\PSL_2(q))
\leqslant 3
\]
which is impossible.

If \(\m_{\overline{G_{vw}}\cap\prod_{i=1}^7\overline{M}_i}(\PSL_2(q))=7\), then 
\[ 4=\m_{G_{vw}}(\PSL_2(q))\geqslant \m_{\overline{G_{vw}}\cap\prod_{i=1}^7\overline{M}_i}(\PSL_2(q))=7,\]
which is also impossible. Thus, the result is proved.
\end{proof}

\begin{lemma}
    Suppose that Hypothesis \ref{hy:1} holds and \(L=E_8(q)\). Then case (15) does not occur.
\end{lemma}
\begin{proof}
    In case (15), $L_v$ has a normal subgroup
    \(
    M\cong \C_{m}^2
    \), where $m=q^4+\epsilon q^3+q^2+\epsilon q+1$ with $\epsilon\in \{1,-1\}$, and $L_v/M\cong 5\times\SL_2(5)$.

Note that $m\mid q^5-\epsilon$.
Let $r\in \mathrm{ppd}(p, {5f(3-\epsilon)}/{2})$.  
   Then $r>5$ and $r\mid m$.
   Let \(R=\mathbf{\Omega}_r(M)\cong \mathrm{C}_r^2\) and consider the conjugation of $G_v$ on $R$, with kernel $C=\C_{G_v}(R)$.
   Then $M\leq C$ and $G_v/C\lesssim \mathrm{GL}_2(r)$.
   We view $G_v/C$ as a subgroup of $\mathrm{GL}_2(r)$ acting on $R$.
    According to \cite[Lemma 4.6(i)]{LSS}, \(L_v/C\) is irreducible on \(R\).
    By maximal subgroups of $\mathrm{SL}_2(r)$ and its automorphism groups in~\cite[Tables 8.1 and 8.2]{holt}, we conclude that $L_v/C$ has an irreducible normal subgroup $S=\mathrm{SL}_2(5) $.
    Note that $S\trianglelefteq G_v/C$ since it is characteristic in $L_v/C$.
    Let $\phi$ be the natural homomorphism from $\mathrm{GL}_2(r)$ to $\mathrm{PGL}_2(r)$.
    By \cite[Table 8.2]{holt},  $\N_{\mathrm{PGL}_2(r)}(\phi(S))=\phi(S)\cong \mathrm{A}_5$.
    This implies that $G_v/C$ acts trivially on $S$ by conjugation, and hence $G_v/C=\mathrm{C}_k\circ S$ for some positive integer $k$ such that $k\mid 5f$.
    Therefore, $G_v$ has a normal subgroup $N$ such that $G_v/N\cong  \phi(S)\cong  \mathrm{A}_5 $.

Let \(F=\mathbb{Q}(\zeta_{60})\), where \(\zeta_{60}\) is a primitive \(60\)-th root of unity in $\mathbb{C}$.
Since $60$ is the exponent of $S$, the field $F$ is a split field of $S$ (see~\cite[Theorem 9.2.7]{Webb}).
Let $V$ be a $2$-dimensional irreducible $FS$-module. 
We consider all factorisations \(S=AB\) with \(A, B\leqslant S\). 
For each such factorisation, we check whether the restriction of \(V\) to \(A\) or to \(B\) is irreducible.
Computation in {\sc Magma} shows that either $A$ or $B$ is irreducible on $V$ for all factorisations $S=AB$ with \(A, B\leqslant S\) and all such $2$-dimensional irreducible $ FS$-modules $V$ (there are exactly two such modules).

Consider the factorisation $G_v/N= (G_{uv}N/N)(G_{vw}N/N)$.
Let $Z\cong \mathrm{C}_2$ be the centre of $S=\mathrm{SL}_2(5)=2.\mathrm{A}_5$, and view $G_v/N\cong \mathrm{A}_5$ as the factor group $S/Z$.
For each pair $(G_{uv}N/N, G_{vw}N/N)$, we let $A,B\leqslant S$ such that $A/Z=G_{uv}N/N$ and $B/Z=G_{vw}N/N$.
Then $S=AB$.  
Since \(r \nmid |S| = 120\) (as $r>5$), the irreducible \(FS\)-module \(V\) reduces modulo \(r\) to an irreducible \(\mathbb{F}_r S\)-module (see~\cite[Lemma 9.4.1]{Webb}).
Therefore, the computation result in the last paragraph implies that at least one of $G_{uv}$ and $G_{vw}$ acts irreducibly on $R$.
    Without loss of generality, we assume that $G_{uv}$ acts irreducibly on $R$. 
     Note that $|L_v/M|=|5\times \mathrm{SL}_2(5)|$ and $G_v/L_v\leqslant \mathrm{C}_f$.
     Thus, $ |G_v/M|_r=1$.
     It follows from $G_v=G_{uv}G_{vw}$ that $|G_{uv}\cap M|_r>1$ and hence $G_{uv}\cap R>1$.
     Then the irreducibility of $G_{uv}$ on $R$ implies $R\leqslant G_{uv}$.
     Thus, $R^g\leqslant G_{vw}$ as $G_{vw}=G_{uv}^g$.
     Since $ |G_v/M|_r=1$,  $R$ is the unique subgroup of $G_v$ isomorphic to $\mathrm{C}_r^2$.
     Therefore, $R^g=R$, contradicting Lemma \ref{factor}. Hence, the result is proved.
   \end{proof}

\begin{lemma}
    Suppose that Hypothesis \ref{hy:1} holds and \(L=E_8(q)\). Then case (16) does not occur.
\end{lemma}
\begin{proof}
    In case (16), $L_v$ has a normal subgroup
    \(
    M\cong \C_{m}^2
    \), where $m=q^4-q^2+1$, and $L_v/M\cong \mathrm{C}_{12}\circ (\mathrm{Q}_8.\mathrm{S}_3)$.
The argument follows the same strategy as in case (15).   Let \(r\in \mathrm{ppd}(p,12f)\) (so that \(r\mid m\)), and consider the conjugation action of \(G_v\) on 
    \(R:=\mathbf{\Omega}_r(M)\cong \C_r^2\). 
    Note that \(\mathrm{Q}_8.\mathrm{S}_3 \cong 2.\mathrm{S}_4\), and \(\mathrm{S}_4\) is self-normalising in \(\mathrm{PGL}_2(r)\) 
    (see \cite[Table 8.1]{holt}). 
    It follows that \(G_v\) has a normal subgroup \(N\) such that 
    \(G_v/N\cong  \mathrm{S}_4\). 
    We then examine factorisations of \(2.\mathrm{S}_4\) and the irreducibility of the two factors on 
    any \(2\)-dimensional irreducible \(\mathbb{Q}(\zeta_{12})\)-module $V$.  
Computation  in {\sc Magma}   shows that for every group factorisation of \(2.\mathrm{S}_4\), at least one factor 
    acts irreducibly on \(V\). 
    Consequently, at least one of \(G_{uv}\) or \(G_{vw}\) is irreducible on \(R\).Without loss of generality, assume that \(G_{uv}\) acts irreducibly on \(R\). Since \(r>12f=|\Out(L)|\), we have \(|G_v|_r=|L_v|_r=|M|_r\). From \(r\in\Pi(G_v)=\Pi(G_{uv})\) it follows that \(|G_{uv}\cap M|_r=|G_{uv}|_r>1\) and therefore \(G_{uv}\cap R\neq 1\). Since the action of \(G_{uv}\) is irreducible on \(R\), it follows that \(R\leqslant G_{uv}\) and so \(R^g\leqslant G_{vw}\). Note that \(R\) is the unique subgroup of \(G_v\) isomorphic to \(\C_r^2\). Thus, \(R^g=R\), contradicting Lemma \ref{factor}.
   \end{proof}

\begin{lemma}
 Suppose that Hypothesis~$\ref{hy:1}$ holds and \(L=E_8(q)\). Then case (17) does not occur.    
\end{lemma}
\begin{proof}
In case (17), 
 $L_v$ has a normal subgroup
    \(
    M\cong \C_{m} 
    \), where $m=q^8+\epsilon q^7-\epsilon q^5-q^4-\epsilon q^3+\epsilon q+1$ with $\epsilon\in \{1,-1\}$, and $L_v/M\cong \mathrm{C}_{30}$.
    Note that $m=\Phi_{15a}(q)$, where $a=(3-\epsilon)/2$.
    Hence, $m$ is divisible by some $r\in\mathrm{ppd}(p,15af)$. 
    Since $r>15af$,  we have \(|G/L|_r=|\Out(L)|_r=f_r=1\), and therefore the result follows from Lemma~\ref{lm:m.o}. 
    \end{proof}

\begin{lemma}
    Suppose that Hypothesis \ref{hy:1} holds and \(L=E_8(q)\). Suppose that case (18) occurs. Then \(\Gamma\) is not \((L,2)\)-arc-transitive and \(s\leqslant2\).
\end{lemma}
\begin{proof}
   Suppose, for a contradiction, that case (18) occurs and \(\Gamma\) is \((L,2)\)-arc-transitive. Then \(L_v=L_{uv}L_{vw}\). Note that \(L_{v}\cong \C_{q^2+1}^{4}.((4\circ2^{1+4}).\Sy_{6})\). Let us denote by \(M\cong\C_{q^2+1}^{4}\) and \(\overline{S}=SM/M\) for any \(S\leqslant L_{v}\). 
   
   Let \(r\in\ppd(q,4)\). Then \(r>4\) and \(r\mid q^{2}+1\). Note that \(R:=\mathbf{\Omega}_r(M)\) is the unique subgroup of \(M\) isomorphic to \(\C_{r}^{4}\), and from \cite[Lemma 4.6(i)]{LSS} it follows that \((4\circ2^{1+4}).\Sy_{6}\) acts irreducibly on \(K\).

   \medskip
   Observe that \(\ICF(G_{v})=\{\A_6\}\), so that \cite[Lemma 3.3]{small} applies, and it follows that \(\ICF(L_{uv})=\ICF(L_{vw})=\{\A_i\}\) for \(i\in\{5,6\}\). Thus, \(\A_i\in\ICF(\overline{L_{uv}})\cap\ICF(\overline{L_{vw}})\) for \(i\in\{5,6\}\). Thus, \(|\overline{L_{uv}}|\) and \(|\overline{L_{vw}}|\) are divisible by \(60\).  Now, by checking {\sc{magma}}\cite{magma} the factorisations of \((4\circ2^{1+4}).\Sy_6\) satisfying the above numerical condition, we conclude that at least one of \(\overline{L_{uv}}\) or \(\overline{L_{vw}}\) is irreducible on \(R\cong\C_r^4\). Without loss of generality, assume that \(\overline{L_{uv}}\) is irreducible.
  
This implies that \(L_{uv}\cap R=1\) or \(R\). Assume that \(L_{uv}\cap R=1\). Then \(|L_{uv}|_{r}=|\varphi(L_{uv})|_{r}\leqslant |\Sy_{6}|_{r}\leqslant r\). This would imply that \(|L_{v}|_{r}\geqslant r^{4}>r^{2}\geqslant |L_{uv}|_{r}\), which is impossible. Hence \(L_{uv}\geqslant R\) and so \(R^{g}\leqslant L_{vw}\). Since \(R\) is the unique subgroup of \(L_{v}\) isomorphic to \(\C_{r}^{4}\), it follows that \(R^{g}=R\). However, this contradicts Lemma \ref{factor}.
\end{proof}
\begin{lemma}
    Suppose that Hypothesis \ref{hy:1} holds and \(L=E_8(q)\). Suppose that case (19) occurs. Then \(\Gamma\) is not \((L,2)\)-arc-transitive and \(s\leqslant2\).
\end{lemma}
\begin{proof}
  Suppose, for a contradiction, that case (19) occurs, and \(\Gamma\) is \((L,2)\)-arc-transitive. Then \(L_v=L_{uv}L_{vw}\).
  Let us denote by \(M=\prod_{i=1}^4M_i\) where \(M_i\cong \C_{q^2\pm q+1}^{4}\) for \(1\leqslant i\leqslant4\). Note that \(\ppd(q,3i)\neq \varnothing\) for \(i=1,2\). Let \(r_i\in\ppd(q,3i)\) for \(i=1,2\) and
  \[
  r:=\begin{cases}
      r_1&\text{if \(M\cong \C_{q^2+q+1}^4\);}\\
      r_2&\text{if \(M\cong \C_{q^2-q+1}^4\).}\\
  \end{cases}
  \]
  Then \(R=\mathbf{O}_r(M)\cong\C_r^4\) and it follows from \cite[Lemma 4.6(i)]{LSS} that \(2.(3\times\SU_4(2))\) acts irreducibly on \(R\). 
  
Let \(\overline{S}=SM/M\) for any \(S\leqslant L_v\). Then \(\overline{L_v}=\overline{L_{uv}}\,\overline{L_{vw}}\cong 2.(3\times \SU_{4}(2))\). Note that \(|M|_s=1\) for \(s\in\{2,5\}\). Since \(2,5\in\pi(\SU_4(2))\subseteq\pi(L_v)=\pi(L_{uv})=\pi(L_{vw})\), it follows that \(|\overline{L_{uv}}|_s=|L_{uv}|_s=|L_{vw}|_s=|\overline{L_{vw}}|_s\) for \(s\in\{2,5\}\). Moreover, \(|\overline{L_{vw}}|_2=|\overline{L_{uv}}|_2=|L_{uv}|_2\geqslant |L_{v}|^{\frac{1}{2}}_2=|2.(3\times \SU_{4}(2))|_2^{\frac{1}{2}}=2^4\). Thus, \(|\overline{L_{uv}}|\) and \(|\overline{L_{vw}}|\) are divisible by \(2^4\cdot5\). By checking the factorisations of \(2.(3\times \SU_{4}(2))\) satisfying the above numerical conditions, we find that at least one of the factors is irreducible. Without loss of generality,
 \(\overline{L_{uv}}\) is irreducible. Hence, \(L_{uv}\cap R=1\) or \(R\). First assume that \(L_{uv}\cap R=1\). Since \(\mathbf{O}_r(L_{uv}\cap M)\leqslant R\), it follows that \(|L_{uv}\cap M|_r=1\). Note that \(q^2\pm q+1\) is not divisible by \(5\) and so \(r\neq 5\). This implies that \(|2.(3\times\SU_4(2))|_r=1\) and so \(|L_{uv}|_{r}=|L_{uv}\cap M|_{r}=1\), contradicting the fact that \(\pi(L_{v})=\pi(L_{uv})\). Thus \(R\leqslant L_{uv}\) and \(R^{g}\leqslant L_{vw}\). Since \(R\) is the unique subgroup of \(L_{v}\) isomorphic to \(\C_{r}^{4}\), it follows that \(R^{g}=R\). However, this contradicts Lemma \ref{factor}.
 
\end{proof}
\begin{lemma}\label{lm:e8case2324}
    Suppose that Hypothesis \ref{hy:1} holds and \(L=E_8(q)\). Then case (20) does not occur. 
\end{lemma}
\begin{proof}
 Suppose, for a contradiction, that case (20) occurs. 
 Now, \(L_{v}=\C_m^{8}.W\), where \(m=q-\epsilon\) with \(\epsilon\in\{1,-1\}\), $W=W(E_8)\cong 2.\mathrm{SO}_8^+(2)$, and \(q\geqslant5\) if \(\epsilon=+1\). 
 Note that  $\vert W\vert= 2^{14} \cdot 3^5 \cdot 5^2 \cdot 7$.
 Take $M$ to be the normal subgroup of $G_v$ such that \(M\cong\C_{q\pm 1}^8\).

Let $\overline{G_{v}}=G_v/\Rad(G_v)$. 
Then $\overline{G_{v}}$ is an almost simple group with socle $\mathrm{P\Omega}_8^{+}(2)$. 
Note that $N_{\mathrm{Aut}(L)}(L_v)=\C_m^{8}.(W\times f)$ (see the proof of \cite[Lemma 3.1]{LSS}).
Therefore, $\overline{G_{v}}\cong \mathrm{SO}_8^+(2)$.
In the factorisation $\overline{G_{v}}= \overline{G_{uv}}\, \overline{G_{vw}}$, since $G_{uv}\cong G_{vw}$, we have \(\ICF(\overline{G_{uv}})=\ICF(\overline{G_{vw}})\).
Computation in {\sc Magma} \cite{magma} shows that both  $\overline{G_{uv}}$ and $\overline{G_{vw}}$ contain a normal subgroup isomorphic to $\mathrm{PSp}_6(2)$ or \(\mathrm{P\Omega}_8^{+}(2)\). 

Let $r\in\pi(m)$ and $R=\Omega_r(M)\cong \mathrm{C}_r^8$.
The group $R$ can be viewed as a $\mathbb{F}_r \mathrm{P\Omega}_8^{+}(2)$-module arising from the action of $W$ on $M$.
Note that \(\POmega_8^+(2)\) acts irreducibly on $R$ for any $r$; On the other hand, if the factorisation $\mathrm{P\Omega}_8^{+}(2)=\overline{G_{uv}}\,\overline{G_{vw}}$ has one factor isomorphic to $\mathrm{PSp}_6(2)$ and \(\ICF(\overline{G_{uv}})=\ICF(\overline{G_{vw}})\), we verify by \textsc{Magma} \cite{magma} that both factors are isomorphic to \(\PSp_6(2)\), and at least one of \(\overline{G_{uv}}\) or \(\overline{G_{vw}}\) acts irreducibly on $R$. 
Thus, without loss of generality, we may assume that  $G_{uv}$ acts irreducibly on $R$.

With a similar argument as in Lemma \ref{lm:e7-qpm1}, we see that \(m=q-\epsilon=2^a \) with $a\geqslant 2$. Let \(r=2\) and \(R=\Omega_2(M)\cong\C_2^8\).  
Since $G_{uv}$ acts irreducibly on $R$, we have $G_{uv}\cap R=1$ or $R$. 

Assume first that $G_{uv}\cap R=R$. Since \(\overline{G_{uv}}\) acts irreducibly on \(R\cong\C_2^8\), it follows that \(R\) is a minimal normal subgroup of \(G_{uv}\). Thus, \(R^g\) is also a minimal normal subgroup of \(G_{uv}^g=G_{vw}\). Since \(G_{vw}\cap M\trianglelefteq G_{vw}\), it follows that \(R^g\leqslant M\) or \(R^g\cap M=1\). If \(R^g\cap M=1\), then \[\C_2^8\cong R^g\cong R^g/(R^g\cap M)\cong R^gM/M\trianglelefteq G_{vw}M/M\leqslant G_v/M\leqslant  W(E_8).f,\]
which is impossible. 
Note that \(\ICF(G_{vw}M/M)=\ICF(\overline{G_{vw}})=\{\PSp_6(2)\}\) or \(\{\POmega^+_8(2)\}\). Thus, \(\PSp_6(2)\leqslant G_{vw}M/M\leqslant \C_2^2\times\Sp_6(2).f\) or \(2.\Omega^+_8(2)\leqslant G_{vw}M/M\leqslant W(E_8).f\). However, none of the above groups contains a normal subgroup isomorphic to \(\C_2^8\), contradicting the claim. 
Thus, \(R^g\leqslant M\). Since \(R\) is the unique subgroup of \(M\) isomorphic to \(R\), it follows that \(R^g=R\), contradicting Lemma \ref{factor}.

Thus, $G_{uv}\cap R=1$. Then  $|G_{uv} \cap M|_2=1$, and so
\[|G_{uv}|_2\leqslant  |G_{uv} \cap M|_2\cdot |W|_2\cdot f_2=|W|_2\cdot f_2=2^{14}f_2.\]
If $f>2$, then $p^f-\epsilon$ is divisible by some $t\in \mathrm{ppd}(p,kf)$, where $k=\frac{3-\epsilon}{2}$.
However, $p^f-\epsilon=2^a$ cannot be divisible by such a prime $t>f>2$. 
Thus, $f\in \{1,2\}$.
Note that $|G_{v}|=|L_v|_2\cdot f_2=2^{8a+14}f_2 $.
Then
\[
\frac{(|G_{uv}|_2)^2}{|G_v|_2}\leqslant \frac{(2^{14}f_2)^2}{|L_v|_2\cdot f_2}=\frac{2^{28}f_2}{2^{8a+14}}\leqslant \frac{2^{28}\cdot 2}{2^{8\cdot 2+14}}=\frac{1}{2},
\]
 contradicting the condition \(|G_{uv}|_2^2\geqslant |G_v|_2\) for the homogeneous factorisation $G_v=G_{uv}G_{vw}$. Hence, the result is proved.
\end{proof}

\medskip\noindent
\emph{Proof of Proposition \ref{prop: e8 max rk}:} This immediately follows from Lemmas \ref{lm:e8case1-6}--\ref{lm:e8case2324}.

\subsection{Other maximal subgroups}\label{ome}
In this section, we deal with the remaining maximal subgroups \(H\) of \(E_7(q)\) and \(E_8(q)\).
\subsubsection{\(L=E_7(q)\)}\label{ome7}
By \cite{Craven} and \cite{mk}, if \(H\) is a maximal subgroup of \(G\) such that \(\Soc(G)=L=E_{7}(q)\) such that \(H\) is neither a parabolic nor maximal rank subgroup, then \(H\) is one of the following:
\begin{enumerate}
    \item [(i)] \(H\cap L=\PSL_2(q)\times F_4(q)\);
    \item[(ii)] \(H\cap L=G_2(q)\times \PSp_6(q)\);
    \item[(iii)] \(H\cap L=\PGL_3^\epsilon(q).2\) with \(\epsilon\in\{-1, 1\}\) and \(p=5\);
    \item[(iv)] \(H\cap L=\N_L(\PGL_3^\epsilon(q).2)\) with \(\epsilon\in\{-1, 1\}\) and \(p\geqslant7\);
    \item[(v)] \(H\cap L=\PGL_3^\epsilon(q).2\) with \(\epsilon\in\{-1, 1\}\) and \(p\geqslant7\);
    \item[(vi)] \(H\cap L=\PSL_2(q)\) with \(p\geqslant17\);
    \item[(vii)] \(H\cap L=\PSL_2(q)\) with \(p\geqslant 19\);
    \item[(viii)] \(H\cap L={}^3\!D_4(q).3\) with \(p\geqslant3\);
    \item[(ix)] \(H\cap L=E_7(q^{\frac{1}{r}}).(r,d)\) where \(d=(2,q-1)\) and \(r\) is a prime;
    \item[(x)] \(H\cap L=(2^2\times\POmega_8^+(q).2^2).T\) where \(T=\Sy_3\) if \(q\equiv1\pmod{8}\) and \(T=\C_3\) if \(q\equiv \pm3\pmod{8}\);
    \item[(xi)] \(H\cap L=\PSL_2(q)\times\PGL_2(q)\).
\end{enumerate}
\begin{lemma}\label{ome7i}
    Suppose that Hypothesis \ref{hy:1} holds and \(L=E_{7}(q)\).Then cases (i)--(ii) do not occur. 
\end{lemma}
\begin{proof}
    Suppose, for a contradiction, that one of the cases (i)--(ii) occurs. If \(q=2\) and case (ii) occurs, then \(G=L\) and \(G_v=G_2(2)\times\PSp_6(2)\).  We check by \textsc{Magma} \cite{magma} that there is no homogeneous factorisation for \(L_v=L_{uv}L_{vw}\), which is a contradiction. Thus, we may assume that if case (ii) occurs, then \(q\geqslant3\).
 
 Note that \(\ppd(p,fi)\neq\varnothing\) for \(i\in\{4,12\}\). For each such \(i\), choose \(r_i\in\ppd(p,fi)\). Note that \(G_v\) has a normal subgroup \(M\) such that \(r\in\pi(M)\). Let \(C=\C_{G_v}(M)\) and \(\overline{S}:=SC/C\) for any \(S\leqslant G_v\). Then \(\overline{G_v}\) is almost simple with socle \(T\), where \[(M,r,\N_{\Aut(L)}(C),T)=\begin{cases}(F_4(q),r_{12},\PSL_2(q).\Out(L),F_4(q))&\text{case (i);}\\(\PSp_6(q),r_{6},G_2(q).\Out(L),\PSp_6(q))&\text{case (ii).}\end{cases}\] 

For both cases, note that for \(r\in J\), since \(|\Out(L)|=(2,q-1)f\) and \(r>4f\), we have \(|\Out(L)|_r=1\) and therefore \(|C|_r=|\N_{\Aut(L)}(C)|_r=1\). Thus, \(|\overline{G_{uv}}|_{r}=|\overline{G_{vw}}|_{r}>1\). Then it follows from \cite[Tables 1--3, 5]{LPS} that at least one of \(\overline{G_{uv}}\) or \(\overline{G_{vw}}\) contains \(T\). Without loss of generality, assume that \(\overline{G_{uv}}\geqslant T\).

Thus, in both cases, \(T\leqslant \overline{G_{uv}}\), it follows that \(M\trianglelefteq G_{uv}\). Since \(G_{uv}^{g}=G_{vw}\), we deduce that \(M^{g}\trianglelefteq G_{vw}\). However, \(G_v\) contains no other subgroup isomorphic to \(M\). Hence, \(M^{g}=M\), contradicting Lemma~\ref{factor}. Thus, the result is proved.
\end{proof}
\begin{lemma}
    Suppose that Hypothesis \ref{hy:1} holds and \(L=E_7(q)\). Then cases (iii)--(x) do not occur.
\end{lemma}

\begin{proof}
For these cases, $G_v^{(\infty)}$ is quasisimple. Let $S$ be the socle of $G_v/\mathrm{R}(G_v)$. Then  $S $ is one of the following:
\[ \PSL_3^\epsilon(q) \text{ with \(p \geqslant 5\) }, \PSL_2(q) \text{ with \(p \geqslant 17\) }, {}^3\!D_4(q), E_7(q^{\frac{1}{r}}) \text{ with \(r\) prime }, \POmega_8^+(q). \]

Applying Lemma~\ref{lm:qsimple}(a), we only need to consider the groups $\PSL_2(23) $ and $ \POmega_8^+(q)$.

Assume $S=\PSL_2(23)$.
By~\cite[Table 4.1]{Craven}, $L_v=\PSL_2(23)$. Now, $G/L\leqslant \mathrm{Out}(L)=2$.
Hence, $\mathrm{R}(G_v)=1$ or $2$.
If $\mathrm{R}(G_v)=1$, then $G_v$ is almost simple with socle $\mathrm{PSL}_2(23)$, which is impossible by Lemma~\ref{lm:qsimple}(a).
Let $\mathrm{R}(G_v)=2$. Then $G_v/\mathrm{R}(G_v)=\mathrm{PSL}_2(23)$.
By Table~\ref{tb:exceptionalfacs}, we see that $G_{uv}R/\mathrm{R}(G_v)=23:3$ and $\G_{vw}R/\mathrm{R}(G_v)=\mathrm{S}_4$.
This implies that $|G_{uv}|_{23}=23>1=|G_{vw}|_{23}$, contradicting $G_{uv}\cong G_{vw}$.

Assume $S=\POmega_8^+(q)$.
Then $L_v=(2^2\times\POmega_8^+(q).2^2).T$, where \(T=\Sy_3\) if \(q\equiv1\pmod{8}\) and \(T=\C_3\) if \(q\equiv \pm3\pmod{8}\). Note that \(T\) always contain an element \(\tau\) such that \(|\tau|=3\).
By~\cite[p. 35]{CLSS}, \(\tau\) acts as a graph automorphism of $\POmega_8^+(q)$.
Hence, $G_v/\mathrm{R}(G_v) $ is not a subgroup of $\mathrm{P\Gamma O}_8^{+}(q) $, contradicting Lemma~\ref{lm:qsimple}(b.2). 
\end{proof}

\begin{lemma}\label{om7xi}
    Suppose that Hypothesis \ref{hy:1} holds and \(L=E_{7}(q)\). Suppose that case (xi) occurs. Then \(s\leqslant 2\).
\end{lemma}
\begin{proof}
Suppose, for a contradiction, that case (xi) occurs. Then \(L_v\cong M_1\times M_2\) such that \(M_1\cong\PSL_2(q)\) and \(M_2\cong\PGL_2(q)\) with \(q\geqslant 5\). We will split our analysis into three cases.

\medskip
\emph{Case (1):} \(q\) is even.\quad
Then \(p=2\), \(f\geqslant3\) and \(M_1\cong M_2\cong\PSL_2(q)\). If \(f=3\) or \(6\), then \((|L_v|,|\Out(L)|)=(2^6\cdot3^4\cdot7^2,3)\) or \((2^{12}\cdot3^4\cdot5^2\cdot7^2\cdot13^2,6)\), respectively.
 From Lemma \ref{3ATprime} it follows that \(|L_{uv}|\) is divisible by \(2^4\cdot3^3\cdot7^2\) for \(f=3\) and \(2^8\cdot3^3\cdot5^2\cdot7^2\cdot13^2\) for \(f=6\). Then, by \textsc{Magma} \cite{magma} computation, there is no homogeneous factorisation for \(L_v=L_{uv}L_{vw}\) satisfying the numerical condition above.

Now, \(f\neq 3,6\) and therefore \(\ppd(2,fi)\neq \varnothing\) for \(i\in\{1,2\}\). Let \(r_i\in\ppd(2,fi)\) for \(i\in\{1,2\}\). Then \(|\PSL_2(q)|_{r_i}\), \(r_i>fi\) and \(|\Out(L)|_{r_i}=f_{r_i}=1\). We claim that \(|L_{uv}\cap M_i|_{r_j}>1\) for any \(i,j\in\{1,2\}\). Suppose that this assertion is false and there are \(i,j\in \{1,2\}\) such that \(|L_{uv}\cap M_i|_{r_j}=1\). Then
\[
|L_{uv}|_{r_j}\leqslant |L_{uv}\cap M_i|_{r_j}|M_{3-i}|_{r_j}=|M_{3-i}|_{r_j}=|\PSL_2(q)|_{r_j}.
\]

This together with Lemma \ref{3ATprime} implies that
\[
|\PSL_2(q)|^4_{r_j}=|L_v|_r^2\leqslant|L_{uv}|_r^3|\Out(L)|_r=|\PSL_2(q)|_{r_j}^3,
\]
which is impossible. Thus, \(|L_{uv}\cap M_i|_{r_j}>1\) for \(i,j\in\{1,2\}\) and so \(L_v'=\prod_{i=1}^2\Soc(M_i)\leqslant L_{uv}\). This implies that \(|L_v:L_{uv}|=2\), contradicting the condition that the valency of \(\Gamma\) is at least \(3\).

\medskip
\emph{Case (2):} \(q=p=2^a-1\) is a Mersenne prime.\quad
Then \(|\Out(L)|=2\) and \(a\geqslant 5\).  
We claim that \(|L_{uv}\cap M_i|_2\geqslant 2^2\) and \(|L_{uv}\cap M_i|_p>1\) for \(i\in\{1,2\}\). Assume that the assertion is false and \(|L_{uv}\cap M_i|_{2}\leqslant2\) and \(|L_{uv}\cap M_i|_p=1\). Then, for \(r\in\{2,p\}\),
\[|L_{uv}|_r\leqslant |L_{uv}\cap M_i|_r|M_{3-i}|_r\leqslant (2|\PGL_2(q)|)_r\leqslant\begin{cases}
    2^{a+1}&\text{\(r=2\)};\\
    p^{f}&\text{\(r=p\).}
\end{cases}
\]
Then, it follows from Lemma \ref{3ATprime} that
\[
2^{4a+2}=(|\PSL_2(q)|_2|\PGL_2(q)|_2)^2=|L_v|_2^2\leqslant|L_{uv}|^3_2|\Out(L)|_2\leqslant 2^{3a+4}
\]
and  
\[
p^{4f}=(|\PSL_2(q)|_p|\PGL_2(q)|_p)^2=|L_v|_p^2\leqslant|L_{uv}|^3_p|\Out(L)|_p\leqslant p^{3f},
\]
leading to the conclusion that \(a\leqslant 2\) and \(f=0\). This is a contradiction, and so the assertion is correct. Now, since \(|M_i:\Soc(M_i)|\leqslant |\PGL_2(q):\PSL_2(q)|=2\), it follows that \(|L_{uv}\cap\Soc(M_i)|_p=|L_{uv}\cap\Soc(M_i)|_p>1\) and \(|L_{uv}\cap\Soc(M_i)|_2\geqslant 2^{-1}|L_{uv}\cap\Soc(M_i)|_2\geqslant2\) for \(i\in\{1,2\}\). By \cite[Table 10.3]{transitive} we deduce that \(L_{v}'=\prod_{i=1}^2\Soc(M_i)\leqslant L_{uv}\) and so \(|L_{v}:L_{uv}|\leqslant2\), contradicting the condition that the valency of \(\Gamma\) is at least \(3\).

\medskip
\emph{Case (3):} \(q\) is even and \(q\) is not a Mersenne prime.\quad For \(q=5\) and \(9\), we have \((|L_v|,|\Out(L)|)=(2^5\cdot3^2\cdot5^2,2)\) and \((2^7\cdot3^4\cdot5^2,4)\), respectively. It follows from Lemma \ref{3ATprime} that \(|L_{uv}|\) is divisible by \(2^3\cdot3^2\cdot5^2\) and \(2^4\cdot3^3\cdot5^2\) for \(q=5\) and \(9\), respectively. By \textsc{magma} \cite{magma} computation, there is no homogeneous factorisation \(L_v=L_{uv}L_{vw}\) with \(L_v\cong\PSL_2(5)\times\PGL_2(5)\) satisfying the above numerical condition and the only homogeneous factorisation \(L_v=L_{uv}L_{vw}\) for \(L_v\cong\PSL_2(9)\times\PGL_2(9)\), satisfying the condition \(2^4\cdot3^3\cdot5^2\mid|L_{uv}|\), occurs when \(L_{uv}=M_1\times \PSL_2(5).2\) and \(L_{vw}=M_1\times\PSL_2(5).2\). Note that \(M_1^g\leqslant L_{vw}\) and since \(M_1\) is the unique subgroup of \(L_v\) isomorphic to \(\PSL_2(9)\), it follows that \(M_1^g=M_1\), contradicting Lemma \ref{factor}. 

Thus, \(q\neq 5,9\). Note that \(\ppd(q,2)\neq \varnothing\). Let \(q_2\in\ppd(q,2)\). For \(r\in\{p,q_2\}\), we claim that \(|L_{uv}\cap M_i|_r>1\) for \(i\in\{1,2\}\). Assume that the assertion is false and that \(|L_{uv}\cap M_i|_r>1\) for \(i\in\{1,2\}\). Then
\[
|L_{uv}|_r\leqslant |L_{uv}\cap M_i|_r|M_{3-i}|_r=|\PGL_2(q)|_r=\begin{cases}p^f&\text{\(r=p\);}\\
(q^2-1)_{q_2}&\text{\(r=q_2\)}.
\end{cases}
\]

Note that \(|\Out(L)|=2f\). This together with Lemma \ref{3ATprime} implies that
\[
p^{4f}=|L_v|_p^2\leqslant|L_{uv}|_p^3|\Out(L)|_p\leqslant p^{3f}(f)_p\leqslant p^{3f}(f!)_p<p^{4f}
\]
and 
\[
(q^2-1)^2_r=|L_v|_r^2\leqslant|L_{uv}|_r^3|\Out(L)|_r\leqslant (q^2-1)_r(f)_r=(q^2-1)_r,
\]
which is impossible. Thus, the result is proved.
\end{proof}
\subsubsection{\(L=E_8(q)\)}\label{ome8}
By \cite[Theorem 2]{LS}, \cite[Theorem 1(II)]{CLSS} and \cite{C2022}, if \(H\) is a maximal subgroup of \(G\) such that \(\Soc(G)=L=E_{8}(q)\) such that \(H\) is neither a parabolic subgroup nor a maximal-rank subgroup, then \(H\) is one of the following:
\begin{enumerate}
    \item [(i)] \(H=(\A_5\times\Sy_6).2\) or \(\Sy_5\times\PGL_2(q)\) with \(p>5\);
    \item[(ii)] \(H\cap L=\PSL_2(q)\times\PSL_3^\epsilon(q)\) with \(\epsilon\in\{\pm\}\) and \(p>3\);
    \item[(iii)] \(H\cap L=G_2(q)\times F_4(q)\);
    \item[(iv)] \(H\cap L=\PSL_2(q)\times G_2(q^2)\) with \(p>2\) and \(q>3\);
    \item[(v)] \(H\cap L=\PSL_2(q)\times G_2(q)\times G_2(q)\) with \(p>2\) and \(q>3\);
    \item[(vi)] \(H\cap L=E_8(q^{\frac{1}{r}})\) where \(r\) is a prime;
    \item[(vii)] \(H\cap L=F_4(q)\) where \(q=3^f\);
    \item[(viii)] \(H=5^3.\SL_3(5)\);
    \item[(ix)] \(H=2^{5+10}.\SL_5(2)\).
\end{enumerate}

\begin{lemma}\label{lm:e8i}
    Suppose that Hypothesis \ref{hy:1} holds and \(L=E_{8}(q)\). Then case (i) does not occur.
\end{lemma}
\begin{proof}
    We check by \textsc{Magma}\cite{magma} that \((\A_5\times\Sy_6).2\) admits no homogeneous factorisation. Now, suppose for a contradiction that case (i) occurs. Then \(G_v=\Sy_5\times\PGL_2(q)\). Let \(M\) denote the subgroup of \(G_v\) isomorphic to \(\PGL_2(q)\). Let \(\varphi\) be the projection from \(G_v\) to \(M\). Since \(p>5\), it follows that \(|\Sy_5|_p=1\), and so \(1<|G_{uv}|_p=|G_{uv}\cap M|_p\leqslant |\varphi(G_{uv})|_p\) and \(1<|G_{vw}|_p=|G_{vw}\cap M|_p\leqslant |\varphi(G_{vw})|_p\). On the other hand, \(\PGL_2(q)=\varphi(G_v)=\varphi(G_{uv})\varphi(G_{vw})\). By checking \cite[Tables 1 and 3]{LPS}, there is no core-free factorisation of \(\PGL_2(q)\) such that the orders of both factors are divisible by \(p\). Hence, one of \(\varphi(G_{uv})\) or \(\varphi(G_{vw})\) contains \(\PSL_2(q)\). Without loss of generality, assume that \(\PSL_2(q)\leqslant \varphi(G_{uv})\). Note that \(G_{uv}\cap M\cong\varphi(G_{uv}\cap M)\trianglelefteq\varphi(G_{uv})\). Thus, either \(G_{uv}\cap M=1\) or \(\PSL_2(q)\leqslant G_{uv}\cap M\). Since \(|G_{uv}\cap M|_p>1\), it follows that \(\PSL_2(q)=\Soc(M)\leqslant G_{uv}\cap M\). Thus, \(\Soc(M)^g\leqslant G_{vw}\). Note that \(\Soc(M)\) is the unique subgroup of \(G_v\) isomorphic to \(\PSL_2(q)\). Therefore, \(\Soc(M)^g=\Soc(M)\), which is a contradiction to Lemma \ref{factor}. Hence, the result is proved.
\end{proof}

\begin{lemma}
    Suppose that Hypothesis \ref{hy:1} holds and \(L=E_{8}(q)\). Then cases (ii)--(iv) do not occur.
\end{lemma}
\begin{proof}
    Suppose for a contradiction that \(s\geqslant3\) and \(L_v\) is one of types (ii)--(iv). Then \(\Gamma\) is \((L,2)\)-arc-transitive. Note that \(\ppd(p,if)\neq \varnothing\) for \(i\in\{3,6,12\}\). Let \(r_i\in\ppd(p,if)\) for \(i\in\{3,6,12\}\). Then  \(L_v=M_1\times M_2\) with \(r\in\pi(M_1)\) falls into one of the following:
    \begin{enumerate}
        \item Case(ii), \((M_1,M_2,r)=(\PSL_3^+(q),\PSL_2(q),r_{3})\);
        \item Case (ii), \((M_1,M_2,r)=(\PSL_3^-(q),\PSL_2(q),r_{6})\);
        \item Case (iii), \((M_1,M_2,r)=(F_4(q),G_2(q),r_{12})\);
        \item Case (iv), \((M_1,M_2,r)=({}^2G_2(q^2),\PSL_{2}(q),r_{12})\).
    \end{enumerate}
           
  Since \(|G_v:L_v|_r\leqslant|\Out(L)|_r=f_r=1\), it follows that \(|G_v|_r=|L_v|_r\). On the other hand, note that \(|M_2|_r=1\) and this implies that \(1<|G_{uv}|_r=|L_{uv}|_r=|L_{uv}\cap M_1|_r\) and \(1<|G_{vw}|_r=|L_{vw}|_r=|L_{vw}\cap M_1|_r\). Let \(\varphi_i\) be the projection from \(L_v\) to \(M_i\) for \(i=1,2\). Then \(1<|L_{uv}\cap M_1|_r\leqslant|\varphi_1(L_{uv})|_r\) and \(1<|L_{vw}\cap M_1|_r\leqslant |\varphi_1(L_{vw})|_r\). Note that \(M_1=\varphi_1(L_v)=\varphi_1(L_{uv})\). By checking \cite[Table 5]{LPS}, we find that there is no core-free factorisation of \(M_1\) such that the orders of the both factors are divisible by \(r\). Hence, one of \(\varphi_1(L_{uv})\) or \(\varphi_1(L_{vw})\) contains \(M_1\). Without loss of generality, assume that \(M_1\leqslant \varphi(L_{uv})\). Note that \(L_{uv}\cap M_1\cong\varphi_1(L_{uv}\cap M_1)\trianglelefteq\varphi(L_{uv})\). Thus, either \(L_{uv}\cap M_1=1\) or \(M_1\leqslant L_{uv}\). Since \(|L_{uv}\cap M_1|_r>1\), it follows that \(M_1\leqslant L_{uv}\). Thus, \(M_1^g\leqslant L_{vw}\). Note that there is no other subgroup of \(G_v\) isomorphic to \(M_1\). Therefore, \(M_1^g=M_1\), which is a contradiction to Lemma \ref{factor}. Hence, the result is proved.

\end{proof}
\begin{lemma}
    Suppose that Hypothesis \ref{hy:1} holds and \(L=E_{8}(q)\). Suppose that case (v) occurs. Then \(s\leqslant2\).
\end{lemma}
\begin{proof}
    Suppose for a contradiction that \(s\geqslant3\) and case (v) occurs. Then \(\Gamma\) is \((L,2)\)-arc-transitive and \(L_v=G_2(q)\times G_2(q)\times \PSL_2(q)\) with \(p>2\) and \(q>3\). Let \(M_1\cong M_2\cong G_2(q)\) and \(M_3\cong\PSL_2(q)\) such that \(L_v=\prod_{i=1}^3M_i\). Note that \(\ppd(p,if)\neq \varnothing\) for \(i\in\{3,6\}\). Let \(r_i\in\ppd(p,if)\) for \(i\in\{3,6\}\) and \(r_0=p\). Then \(|M_1|_{r_j}=|M_2|_{r_j}>1\) and \(|\Out(L)|_{r_j}|M_3|_{r_j}=1\) for \(j\in\{3,6\}\). Let \(\varphi_i\) be the projection from \(L_v\) to \(M_i\) for \(i\in\{1,2,3\}\).

    \medskip
    \textsc{Claim:} \(|L_{uv}\cap M_i|_{r_j}>1\) and \(|L_{vw}\cap M_i|_{r_j}>1\) for \(i\in\{1,2\}\) and \(j\in\{0,3,6\}\).\quad

    Suppose that the assertion is false. Then there exists \(i\in\{1,2\}\) and \(j\in\{0,3,6\}\) such that \(|L_{uv}\cap M_i|_{r_j}=1\). Then 
    \[|L_{uv}|_{r_j}\leqslant |L_{uv}\cap M_i|_{r_j}|\varphi_{3-i}(L_{uv})|_{r_j}|M_3|_{r_j}\leqslant|M_{3-i}|_{r_j}|M_3|_{r_j}\leqslant \begin{cases}
        |G_2(q)|_{r_j}&\text{if \(j\in\{3,6\}\);}\\
        p^{7f}&\text{if \(j=0\).}
    \end{cases}\]
   If \(j\in\{3,6\}\), then by Lemma \ref{3ATprime}, we have 
    \[
    |G_2(q)|_{r_j}^4=|L_{v}|^2_{r_j}\leqslant |L_{uv}|_{r_j}^3|\Out(L)|_{r_j}\leqslant |G_2(q)|_{r_j}^3,
    \]
    which is a contradiction since \(|G_2(q)|_{r_j}\geqslant1\).
    
    If \(j=0\), then \(r_0=p\) and it follows from Lemma \ref{3ATprime} that
    \[
    (f!)_p\geqslant f_p=|\Out(L)|_p\geqslant |L_{v}|_p^2|L_{uv}|_p^{-3}\geqslant p^{26f}p^{-21f}=p^{5f}.
    \]
    However, this contradicts Lemma \ref{sizeppd}. Hence, for \(i\in\{1,2\}\) and \(j\in\{0,3,6\}\), \(|L_{uv}\cap M_i|_{r_j}>1\) and with the same argument we conclude that \(|L_{vw}\cap M_i|_{r_j}>1\). Therefore, the claim is proved.

    Now, since \(|L_{uv}\cap M_i|_{r_j}>1\) for \(j\in\{3,6\}\), it follows from \cite[Table 10.5]{transitive} that either \(\{r_3,r_6\}=\{7,13\}\) and \(L_{uv}\cap M_i=\PSL_2(13)\), or \(M_i\leqslant L_{uv}\). If the former case occurs, then since \(|L_{uv}\cap M|_{p}=2^2\cdot3\cdot7\cdot13>1\) and \(p>2\), it follows that \(p=3\). Since \(q=3^f>3\), we deduce that \(f\geqslant 2\). If \(f\geqslant3\), then \(r_6>6f>18\) and so \(r_6\notin\{7,13\}\). Thus, \(f=2\) and \(\ppd(3,12)=\{73\}\). This again implies that \(73=r_6\notin\{7,13\}\), which is a contradiction.

    Thus, the latter case occurs and \(M_i\leqslant L_{uv}\) for \(i\in\{1,2\}\). Therefore, \(M_1\times M_2\leqslant L_{uv}\) and \((M_1\times M_2)^g\leqslant L_{vw}\). However, \(M_1\times M_2\) is the unique subgroup of \(L_v\) isomorphic to \(G_2(q)\times G_2(q)\) and thus \(M_1\times M_2=(M_1\times M_2)^g\), contradicting Lemma \ref{factor}.
    \end{proof}

   \begin{lemma}\label{lm:e8viii}
    Suppose that Hypothesis \ref{hy:1} holds and \(L=\E_{8}(q)\), then cases (vi)--(ix) do not occur.
\end{lemma} 
\begin{proof}
For cases (vi) and (vii), $L_v=E_8(q^{\frac{1}{r}})$ or \(F_4(q)\), and so $G_v^{(\infty)}$ is quasisimple. By Lemma~\ref{lm:qsimple}, these two cases cannot occur.

Assume that case (viii) occurs.
Then $L_v=5^3.\mathrm{SL}_3(5)$, and by~\cite[Table 1]{CLSS}, $L=E_8(p^f)$ with $p\neq 2,5$, and $a=1$ if $5\mid p^f-1$ or $f=2$ if $5\mid p^f+1$.
Moreover, as noted in~\cite[p. 23]{CLSS}, the group $ 5^3.\mathrm{SL}_3(5)$ can be embedded into $\mathrm{PSL}_4(5)$.
We therefore construct $L_v$ in {\sc Magma} \cite{magma} as a subgroup of $\mathrm{PSL}_4(5)$.
Note that $G/L=a\leqslant 2$.
By~\cite[Lemma 3.2]{small}, $|L_{uv}L_{vw}|=|L_v|$ or $|L_{v}|/2$.
We compute all subgroups pairs $(A,B)$ of the group $L_v= 5^3.\mathrm{SL}_3(5)$ such that $A\cong B$, and $A$ is not conjugate to $B$ in $L_v$, and $|L_v|/|A|\geqslant 3$ (as the valency is at least $3$), and $|AB|=|L_v|$ or $|L_v|/2$.
Computation in {\sc Magma} \cite{magma} shows that there is no such pair $(A,B)$. Therefore, the case (vii) cannot occur.

Now, assume that case (ix) occurs.
Then by~\cite[Table 1]{CLSS}, \(q=p\) and so $L=E_8(p)$ and $ L_v=2^{5+10}.\mathrm{SL}_5(2)$. 
Hence, $G=L$ and $G_v=L_v$.
Let \(E=2^5\) and $E_1=2^{5+10}=\C_{G_v}(E)$.
By~\cite[p. 40]{CLSS}, $E_1$ is a special $2$-group, and $E_1/E$ is the skew-square of a $5$-dimensional module of $\mathrm{SL}_5(2)$.

In the factorisation $\mathrm{SL}_5(2)=G_{v}/E_1=(G_{uv}E_1/E_1) (G_{vv_{1}}E_1/E_1)$, two factors must have orders divisible by $31$ (which is in $\mathrm{ppd}(2,5)$).
By checking the maximal factorisations of almost simple groups with socle $\mathrm{PSL}_5(2)$, see ~\cite{LPS}, we conclude that  $G_{uv}E_1/E_1=G_{vw}E_1/E_1\cong \mathrm{SL}_5(2)$. Note that \(\mathbf{O}_2(G_{uv})=G_{uv}\cap E_1\) and therefore $G_{uv}=\mathbf{O}_2(G_{uv}).\mathrm{SL}_5(2)$. 
Since $\vert G_v\vert$ divides $\vert G_{uv}\vert^2$, we have $\vert G_{uv}\vert_2\geqslant 2^{13}$, which implies that $\vert \mathbf{O}_2(G_{uv})\vert \geqslant 2^3$.

Suppose that $\mathbf{O}_2(G_{uv}) \cap E = 1$. 
Since  $\mathrm{SL}_5(2)$ acts irreducibly on $E_1/E=2^{10}$, we conclude that $G_{uv}E/E\cong 2^{10}.\mathrm{SL}_5(2)$ and subsequently $\mathbf{O}_2(G_{uv}) =2^{10}\lesssim E_1/E$ is abelian. 
Now, $E_1\geqslant E\mathbf{O}_2(G_{uv})=E\times \mathbf{O}_2(G_{uv})=2^5\times 2^{10}$. By comparing the orders, we identify \(E_1=E\times \mathbf{O}_2(G_{uv})\). However, this would imply that \(E_1\) is abelian, which is a contradiction.

Therefore, $\mathbf{O}_2(G_{uv}) \cap E \neq 1$.
Since $\mathrm{SL}_5(2)$ acts irreducibly on $E$ (note that $\mathrm{SL}_5(2)\cong G_v/\C_{G_v}(E)=G_v/E \lesssim \mathrm{Aut}(E)\cong \mathrm{SL}_5(2) $), we have $E\leqslant G_{uv}$.
If $\mathbf{O}_2(G_{uv})>E$, then the irreducibility of $\mathrm{SL}_5(2)$  on  $E_1/E$ implies $E_1<G_{uv}$ and hence $G_{uv}=G_{v}$, a contradiction.
Therefore, $\mathbf{O}_2(G_{uv})=E$. 
By similar arguments to $G_{vw}$ we also have $E=\mathbf{O}_2(G_{vw})$.
Since $G_{uv}^g=G_{vw}$, $E^g$ is also a normal  $2$-subgroup of $G_{vw}$ and hence $E^g=\mathbf{O}_2(G_{vw})=E$, contradicting Lemma~\ref{factor}. 
\end{proof}
\medskip\noindent
\emph{Proof of Theorem \ref{mainthm}:} According to Lemma~\ref{lm:qsimple}(a), \(\Gamma\) is at most \(2\)-arc-transitive for \(G_v\in\mathcal{S}\). Moreover, from Lemma \ref{lm:para} it follows that if \(G_v\) is a maximal parabolic subgroup, then \(s\leqslant 1\). Thus, it suffices to consider \(G_v\) so that \(G_v\) is neither in \(\mathcal{S}\) nor parabolic. For the remaining subgroups, if \(G_v\) is a maximal rank subgroup, then by Propositions \ref{max rk 7} and \ref{prop: e8 max rk} that \(s\leqslant2\). If \(G_v\) is not of maximal rank, then \(G_v\) falls into cases (i)--(xi) given in \ref{ome7} for \(L=E_7(q)\), and cases (i)--(ix) given in \ref{ome8} for \(L=E_8(q)\), and these cases were analysed from Lemmas \ref{ome7i} to \ref{lm:e8viii}. Hence, the result follows.

\appendix
\section{Possible \(2\)-arc-transitive digraphs}
In this section, we list the possible $L_v$ for \(G\)-vertex-primitive \((G,2)\)-arc-transitive digraphs with $L=\mathrm{Soc}(G)$ a simple exceptional group of Lie type in Table \ref{tab:Lvcandidate}.
\begin{table}[h]
\centering
\caption{Possible $L_v$ for \(G\)-vertex-primitive \((G,2)\)-arc-transitive digraphs with $L=\mathrm{Soc}(G)$ a simple exceptional group of Lie type. Notation: $d=(2,q-1)$, $\epsilon$ runs over $\{-1,1\}$, and $e_\epsilon=(3,q-\epsilon)$.}\label{tab:Lvcandidate}
\begin{tabular}{lll}
\toprule
 \(L\)  &\(L_v\)& Remarks \\ 
\midrule
 \({}^3\!D_4(q)\)  &\(  (q^{2}+ \epsilon q+1)^{2}:\mathrm{SL}_{2}(3) \)&  $q\geqslant 4$ \\ 
 \(G_2(q)\)  &\( (\mathrm{SL}_{2}(q)\circ \mathrm{SL}_{2}(q)).2 \)&  $q>19$, $q$ odd \\
  &\( \mathrm{SL}_{2}(q)\times \mathrm{SL}_{2}(q)\)&  $q \geqslant 32$, $q$ even \\  
   &\( (q+\epsilon )^{2}.\mathrm{D}_{12}\)&  $q = 3^{f} > 9$ \\  
 \({}^2\!F_4(q)\)  &\( {}^2\!B_2(q)\wr 2 \)& $q\geqslant 8$ \\  
 &\(  (q+1)^2:\mathrm{GL}_2(3) \)&  $q\geqslant 8$ \\ 
   &\( (q+\epsilon \sqrt{2q}+1)^{2}:[96] \)&  $q\geqslant 8$ \\ 
 \( F_4(q) \)  &\( (\mathrm{Sp}_4(q) \times \mathrm{Sp}_4(q)).2 \)&   \\  
   &\( (\mathrm{SL}_3^{\epsilon}(q) \circ \mathrm{SL}_3^{\epsilon}(q)).e_{\epsilon}.2 \)& \\   
   &\( (q^2 +\epsilon q+1)^2.(3 \times \mathrm{SL}_2(3)) \)& $q$ even  \\  
   &\( ( q^2+ 1)^2.(\mathrm{SL}_2(3) : 4) \)&  $q >2$, $q$ even   \\   
   &\(  ( q-1)^4.W(F_4) \)&  $q >4$, $q$ even  \\ 
  \( E_6^{\epsilon}(q)  \)   &\(  (q^2+\epsilon q+1)^3/e_{\epsilon}.(3^{1+2}.\mathrm{SL}_2(3)) \)&   \\  
   &\(  (q-\epsilon)^6/e_{\epsilon}.W(E_6) \)& \(q>3+\epsilon\) \\  
 \( E_7(q) \)    &\( \mathrm{M}_{12}  \)& $q=p=5$  \\  
   &\( \Soc(L_v)\cong \mathrm{A}_6 \)&    \\  

 \(E_8(q)\) & \(\Soc(L_v)\cong \mathrm{A}_6\) &  \\ 
  & \((5,q-1).(\PSL^\epsilon_{5}(q))^2.(5,q-1).4\) & \\
& \(d^2.(\POmega_8^+(q))^2.d^2.(\Sy_3\times 2)\) & \\
& \((\PSU_3(q^2))^2.8\) & \\
& \(({}^3\!D_4(q))^2.6\) & \\
& \(e_{\epsilon}^2.(\PSL_3^\epsilon(q))^4.e_{\epsilon}^2.\GL_2(3)\) &  \\
& \((q^2+1)^4.(4\circ2^{1+4}).\A_6.2\) & \(G\neq L\) \\
& \((q^2+\epsilon q+1)^4.(2.(3\times\PSU_4(2))) \) &  \(G\neq L\),  \(q>2\) if \(\epsilon=-1\)  \\

 \bottomrule
\end{tabular}
\end{table}

\begin{thebibliography}{99}


        


\bibitem{alex} A. Borovik, Structure of finite subgroups of simple algebraic groups, Algebra Logika 28 (1989) 249–279 (Russian), English transl., Algebra Logic 28 (1989) 163–182 (1990).

\bibitem{magma} W. Bosma, J. Cannon, C. Playoust, The MAGMA algebra system I: The user language, J. Symb. Comput. 24 (1997) 235–265.

\bibitem{holt} J.N. Bray, D.F. Holt, C.M. Roney-Dougal, The Maximal Subgroups of the Low-Dimensional Finite Classical Groups, Cambridge University Press, Cambridge, 2013.

\bibitem{DRG} A.E. Brouwer, A.M. Cohen, A. Neumaier, Distance-Regular Graphs, Springer, Berlin, 1989.

\bibitem{BK2025}
T. Burness, M. Korhonen, On fixed-point-free involutions in actions of finite exceptional groups of Lie type, J. Lond. Math. Soc. 112 (2025) e70263.



\bibitem{Carter2} R.W. Carter, Simple Groups of Lie Type: Conjugacy Classes and Complex Characters, John Wiley and Sons, London, 1972.


\bibitem{CGP2023} L. Chen, M. Giudici, C.E. Praeger, Vertex-primitive $s$-arc-transitive digraphs admitting a Suzuki or Ree group, European J. Combin. 112 (2023) 103729.

\bibitem{CGP2024} L. Chen, M. Giudici, C.E. Praeger, Vertex-primitive $s$-arc-transitive digraphs of symplectic groups, J. Algebra 667 (2025) 425–479.

\bibitem{alter2} J. Chen, L. Chen, M. Giudici, J.J. Li, C.E. Praeger, B. Xia, Bounding $s$ for vertex-primitive $s$-arc-transitive digraphs of alternating and symmetric groups, Ars Math. Contemp. (2025), DOI: 10.26493/1855-3974.3325.5e3.

\bibitem{CLSS} A.M. Cohen, M.W. Liebeck, J. Saxl, G.M. Seitz, The local maximal subgroups of exceptional groups of Lie type, finite and algebraic, Proc. Lond. Math. Soc. (3) 64 (1992) 21–48.


\bibitem{C2022} D.A. Craven, D.I.~Stewart, A.R. Thomas, A new maximal subgroup of \(E_8\) in characteristic 3, Proc. Amer. Math. Soc. 150 (2022) 52–69.


\bibitem{Craven} D.A. Craven, On the maximal subgroups of \(E_{7}(q)\) and related almost simple groups, preprint (2022), \url{https://arxiv.org/abs/2201.07081}.

\bibitem{example} M. Giudici, C.H. Li, B. Xia, An infinite family of vertex-primitive 2-arc-transitive digraphs, J. Combin. Theory Ser. B 127 (2017) 1–13.

\bibitem{linear} M. Giudici, C.H. Li, B. Xia, Vertex-primitive s-arc-transitive digraphs of linear groups, J. Pure Appl. Algebra 223 (2019) 5455–5483.

\bibitem{quasi} M. Giudici, B. Xia, Vertex-quasiprimitive 2-arc-transitive digraphs, Ars Math. Contemp. 14 (2018) 67–82.

\bibitem{CFSG} D. Gorenstein, R. Lyons, R. Solomon, The Classification of the Finite Simple Groups, Number 3, Math. Surveys Monogr., vol. 40, Amer. Math. Soc., Providence, RI, 1998.

\bibitem{mk} M.~Korhonen, Structure of an exotic 2-local subgroup in E7(q). Journal of Group Theory, 28 (2025), no. 6, 1285-1296.


\bibitem{LS} M.W. Liebeck, G.M. Seitz, Maximal subgroups of exceptional groups of Lie type, finite and algebraic, Geom. Dedicata 35 (1990) 353–387.

\bibitem{LPS} M.W. Liebeck, C.E. Praeger, J. Saxl, The maximal factorizations of the finite simple groups and their automorphism groups, Mem. Amer. Math. Soc. 86 (432) (1990).

\bibitem{LSS} M.W. Liebeck, J. Saxl, G.M. Seitz, Subgroups of maximal rank in finite exceptional group of Lie type, Proc. Lond. Math. Soc. (3) 65 (1992) 297–325.

\bibitem{transitive} M.W. Liebeck, C.E. Praeger, J. Saxl, Transitive subgroups of primitive permutation groups, J. Algebra 234 (2000) 291–361.

\bibitem{Malle} G. Malle, The maximal subgroups of \({}^2F_4(q^2)\), J. Algebra 139 (1991) 52–69.

\bibitem{mortimer} B. Mortimer, The modular permutation representations of the known doubly transitive groups, Proc. Lond. Math. Soc. (3) 41 (1980) 1–20.

\bibitem{alter} J. Pan, C. Wu, F. Yin, Vertex-primitive s-arc-transitive digraphs of alternating and symmetric groups, J. Algebra 544 (2020) 75–91.




\bibitem{Praeger1989} C.E. Praeger, Finite primitive permutation groups: a survey, in: Groups-Canberra 1989, Lecture Notes in Math., vol. 1456, Springer, Berlin, 1990, pp. 63–84.

\bibitem{Praeger} C.E. Praeger, Highly arc-transitive digraphs, European J. Combin. 10 (1989) 281–292.

\bibitem{csaba} C.E. Praeger, C. Schneider, Permutation Groups and Cartesian Decompositions, Cambridge University Press, Cambridge, 2018.

\bibitem{Webb}
P. Webb, \emph{A Course in Finite Group Representation Theory}, 
Cambridge Studies in Advanced Mathematics, vol. 161, 
Cambridge University Press, Cambridge, 2016.


\bibitem{weiss}R. Weiss, The non-existence of 8-transitive graphs, Combinatorica 1 (1981) 309–311.

\bibitem{Wilson} R.A. Wilson, The Finite Simple Groups, Grad. Texts in Math., vol. 251, Springer-Verlag, London, 2009.

\bibitem{small} F.G. Yin, Y.Q. Feng, B. Xia, The smallest vertex-primitive \(2\)-arc-transitive digraphs, J. Algebra 626 (2023) 1–38.

\bibitem{ex1} F.G. Yin, L. Chen, Exceptional groups and the \(s\)-arc-transitivity of vertex-primitive digraphs, I, J. Algebra 694 (2026) 448–483.

\bibitem{Zsigmondy} K. Zsigmondy, Zur Theorie der Potenzreste, Monatsh. Math. Phys. 3 (1892) 265--284.

\end{thebibliography}
\end{document}